\documentclass[a4paper,oneside, reqno]{amsart}
\usepackage{mathrsfs}
\usepackage{amsfonts}
\usepackage{amssymb}
\usepackage{amsxtra}
\usepackage{dsfont}

\usepackage[english,polish]{babel}

\newcommand{\pl}[1]{\foreignlanguage{polish}{#1}}

\usepackage{color}
\usepackage{enumitem}

\theoremstyle{plain}
\newtheorem{theorem}{Theorem}
\newtheorem{conjecture}{Conjecture}
\newtheorem{proposition}{Proposition}[section]

\newtheorem{lemma}[proposition]{Lemma}
\theoremstyle{definition}
\newtheorem{definition}{Definition}[section]

\numberwithin{equation}{section}

\newcounter{thm}

\theoremstyle{plain}

\newcommand{\RR}{\mathbb{R}}
\newcommand{\ZZ}{\mathbb{Z}}
\newcommand{\TT}{\mathbb{T}}
\newcommand{\CC}{\mathbb{C}}
\newcommand{\NN}{\mathbb{N}}

\newcommand{\calF}{\mathcal{F}}

\newcommand{\ind}[1]{{\mathds{1}_{{#1}}}}

\newcommand{\sprod}[2] {{#1 \cdot #2}}

\newcommand{\abs}[1]{{\lvert {#1} \rvert}}

\newcommand{\la}{\lambda}

\title[On the Hardy--Littlewood maximal functions]
{On the Hardy--Littlewood maximal functions in high dimensions: 
Continuous and discrete perspective}

\author{Jean Bourgain}
\address{Jean Bourgain \\
  School of Mathematics\\
  Institute for Advanced Study\\
  Princeton, NJ 08540\\
  USA}
\email{bourgain@math.ias.edu}

\author{Mariusz Mirek}

\address{Mariusz Mirek \\
  Department of Mathematics\\
  Rutgers University\\
Piscataway, NJ 08854\\ USA \&
	Instytut Matematyczny\\
	Uniwersytet \pl{Wroc{\lll}awski}\\
	Plac Grun\-waldzki 2/4\\
	50-384 \pl{Wroc{\lll}aw}\\
	Poland}
\email{mariusz.mirek@rutgers.edu}

\author{Elias M. Stein}
\address{
	Elias M. Stein\\
	Department of Mathematics\\
	Princeton University\\
	Princeton\\
	NJ 08544-100 USA}
\email{stein@math.princeton.edu}

\author{B{\l}a{\.z}ej Wr{\'o}bel}
\address{ B{\l}a{\.z}ej Wr{\'o}bel\\
	Instytut Matematyczny\\
	Uniwersytet \pl{Wroc{\lll}awski}\\
	Plac Grun\-waldzki 2/4\\
	50-384 \pl{Wroc{\lll}aw}\\
	Poland}
\email{blazej.wrobel@math.uni.wroc.pl}

\thanks{ Jean Bourgain was supported by NSF grant DMS-1800640.
Mariusz Mirek was partially supported by the Schmidt Fellowship and
the IAS School of Mathematics and by the National Science Center, Poland grant
DEC-2015/19/B/ST1/01149.  Elias M. Stein was partially supported by
NSF grant DMS-1265524.  B{\l}a{\.z}ej Wr{\'o}bel was partially
supported by the National Science Centre, Poland grant Opus
2018/31/B/ST1/00204}

\begin{document}
 
\selectlanguage{english}

\begin{abstract}
This is a survey article about recent developments in dimension-free estimates for maximal functions corresponding to the Hardy--Littlewood averaging operators  associated with convex symmetric  bodies in $\mathbb R^d$ and $\mathbb Z^d$. 
\end{abstract}

\dedicatory{\Large{Dedicated to Fulvio Ricci on the occasion of his 70th birthday.}}

\maketitle

\section{Introduction}

\label{sec:1}
\subsection{Statement of the results: continuous perspective}
Let $G$  be a convex  symmetric body in $\RR^d$, which is simply a bounded closed and  symmetric convex subset of $\RR^d$ with non-empty interior. In the literature it is usually assumed that a symmetric convex body $G\subset\RR^d$ is  open.
In fact, in $\RR^d$ there is no difference whether we assume $G$ is closed or open, since the boundary of a convex set has Lebesgue measure zero. However, in the discrete case, if $G\cap\ZZ^d$ is considered,  it matters. Therefore, later on  in order to avoid some technicalities,  we will assume that a symmetric convex body $G\subset\RR^d$ is always closed.  

For every $t>0$ and  for every  $x\in\RR^d$ we define the  Hardy--Littlewood averaging operator
\begin{align}
\label{eq:93}
M_t^Gf(x)=\frac{1}{|G_t|}\int_{G_t}f(x-y){\rm d}y \quad \text{for} \quad f\in L^1_{\rm loc}(\RR^d),
\end{align}
where $G_t=\{y\in\RR^d: t^{-1}y\in G\}$ denotes a dilate of the body $G\subset\RR^d$.

For $p\in(1, \infty]$, let $C_p(d, G)>0$ be the best constant
such that   the following maximal inequality 
\begin{align}
\label{eq:109}
   \big\|\sup_{t>0}|M_t^Gf|\big\|_{L^p(\RR^d)}\le C_p(d, G)\|f\|_{L^p(\RR^d)}
\end{align}
holds for every $f\in L^p(\RR^d)$.

The question we shall be concerned with, in this survey, is to decide whether the constant $C_p(d, G)$ can be estimated independently of the dimension $d\in\NN$ for every $p\in(1, \infty]$.  

If $p=\infty$, then \eqref{eq:109} holds with $C_p(d, G)=1$, since $M_t^G$ is an averaging operator. By appealing to a covering argument for $p=1$, and a simple interpolation with $p=\infty$, we can conclude that $C_p(d, G)<\infty$ for every $p\in(1, \infty)$ and for every convex symmetric body $G\subset\RR^d$. However, then the implied upper bound for $C_p(d, G)$ depends on the dimension, since the interpolation with a weak type $(1, 1)$ estimate does not give anything reasonable in these kind of questions, and generally it is better to work with $p\in(1, \infty)$ to obtain any non-trivial result concerning the behavior of $C_p(d, G)$ as $d\to\infty$. 

The problem about estimates of $C_p(d, G)$, as $d\to\infty$, has been extensively studied by several authors for nearly four  decades. The starting point was the work of the third  author \cite{SteinMax}, where, in the case of the Euclidean balls $G=B^2$, it was shown that $C_p(d, B^2)$ is bounded independently of the dimension for every $p\in(1, \infty]$. 
Not long afterwards it was proved by the first author, in \cite{B1} for $p=2$, that $C_p(d, G)$ is bounded by an absolute constant, which is independent of the underlying convex symmetric  body $G\subset\RR^d$. This result was extended in \cite{B2}, and independently by Carbery \cite{Car1}, for all $p\in(3/2, \infty]$. 

It is conjectured that the inequality in \eqref{eq:109} holds for all $p\in(1, \infty]$ and for all convex symmetric bodies $G\subset\RR^d$ with  $C_p(d, G)$  independent of $d\in\NN$. It is reasonable to believe that this is true, since it was
verified for a large class of convex symmetric bodies.

For $q\in[1, \infty]$, let $B^q$ be a $q$-ball in $\RR^d$ defined by 
\begin{align}
\label{eq:88}
\begin{split}
  B^q=\Big\{x\in\RR^d\colon& |x|_q=\Big(\sum_{1\le k\le
  d}|x_k|^q\Big)^{1/q}\le 1\Big\} \quad \text{for} \quad q\in[1, \infty),\\
B^{\infty}&=\{x\in\RR^d\colon|x|_{\infty}=\max_{1\le k\le d}|x_k|\le 1\}.     
\end{split}
\end{align}
For the $q$-balls $G=B^q$  the full range $p\in(1, \infty]$ of dimension-free estimates for $C_p(d, B^q)$ was established by M\"uller  in \cite{Mul1}  (for $q\in [1, \infty)$) and in \cite{B3} (for cubes $q=\infty$) with constants depending only on $q$. More about the current state of the art and papers \cite{B1, B2, B3, Mul1, SteinMax} will be given in Section \ref{sec:2}.

The general case is  beyond our reach at this moment.   However, the approach undertaken in the present article permits us to provide a new simple proof of dimension-free
estimates for the Hardy--Littlewood maximal functions associated with symmetric convex bodies  $G\subset\RR^d$, which independently were the subject of \cite{B2} and \cite{Car1}. We prove the following theorem.
\begin{theorem}
  \label{thm:4}
  Let $p\in(3/2, \infty]$, then there exists a constant $C_p>0$ independent of dimension  $d\in\NN$ and a symmetric convex body $G\subset\RR^d$ such that the constant $C_{p}(d, G)$ defined in \eqref{eq:109} satisfies 
  \begin{align}
    \label{eq:43}
   C_{p}(d, G)\le C_p.
  \end{align}
  Moreover, a dyadic variant of \eqref{eq:43} remains true for  all $p\in(1, \infty]$. More precisely, for every $p\in(1, \infty]$ there exists a constant $C_p>0$ independent of dimension  $d\in\NN$ and a symmetric convex body $G\subset\RR^d$ such that 
  for every  $f\in L^p(\RR^d)$ we have
  \begin{align}
    \label{eq:44}
   \big\|\sup_{n\in\ZZ}|M_{2^n}^Gf|\big\|_{L^p(\RR^d)}\le C_p\|f\|_{L^p(\RR^d)}.
  \end{align}
\end{theorem}
The proof of Theorem \ref{thm:4} will be presented in Section \ref{sec:5} using a new flexible approach, which recently resulted in dimension-free bounds in $r$-variational and jump inequalities corresponding to the operators $M_t^G$ from  \eqref{eq:93}, see \cite{BMSW1} and \cite{MSZ1}. An important feature of this method is that it is also applicable to the discrete settings, see \cite{BMSW3} and \cite{MSZ1}. The method is described in Section \ref{sec:3}, the proof of Theorem \ref{thm:4} is given in Section \ref{sec:5}. Our aim  is to continue the investigations in the discrete settings as well. Similar types of questions were recently studied by the authors \cite{BMSW3} for the discrete analogues of the operators $M_t^G$ in $\ZZ^d$.

\subsection{Statement of the results: discrete  perspective}
For every $t>0$ and  for every $x\in\ZZ^d$ we define  the discrete Hardy--Littlewood averaging operator 
\begin{align}
\label{eq:85}
\mathcal M_t^Gf(x)=\frac{1}{|G_t\cap \ZZ^d|}\sum_{y\in G_t\cap\ZZ^d}f(x-y) \quad \text{for} \quad f\in\ell^1(\ZZ^d).
\end{align}  
We note that the operator $\mathcal M_t^G$ is a discrete analogue of $M_t^G$ from \eqref{eq:93}.

For $p\in(1, \infty]$, let $\mathcal C_p(d, G)>0$ be the best constant
such that the following  maximal inequality 
\begin{align}
\label{eq:86}
   \big\|\sup_{t>0}|\mathcal M_t^Gf|\big\|_{\ell^p(\ZZ^d)}\le \mathcal C_p(d, G)\|f\|_{\ell^p(\ZZ^d)}
\end{align}
holds for every $f\in\ell^p(\ZZ^d)$.

Arguing in a similar way as in \eqref{eq:109} we conclude  that $\mathcal C_p(d, G)<\infty$ for every $p\in(1, \infty]$ and for every convex symmetric body $G\subset\RR^d$. The question now is to decide whether $\mathcal C_p(d, G)$ can be bounded independently of the dimension $d$ for every $p\in(1, \infty)$.

In \cite{BMSW3} the authors examined this question in the case of the discrete cubes $B^{\infty}\cap\ZZ^d$, and showed that for every $p\in(3/2, \infty]$ there is a constant $C_p>0$ independent of the dimension such that $\mathcal C_p(d, B^{\infty})\le C_p$. It was also shown in \cite{BMSW3} that if the supremum in \eqref{eq:86} is restricted to the dyadic set $\mathbb D=\{2^n:n\in\NN\cup\{0\}\}$, then \eqref{eq:86} holds for all $p\in(1, \infty]$ and  $\mathcal C_p(d, G)$ is independent of the dimension.

The general case in much more complicated. However, it is not difficult to show  \cite{BMSW3} that for every symmetric convex body $G\subset\RR^d$ there exists $t_{G}>0$ with the property that the norm of the discrete maximal function $\sup_{t>t_{G}}|\mathcal M_t^Gf|$ is controlled by a constant multiple of the norm of its continuous counterpart, and the implied constant is independent of the dimension.
This is a simple comparison argument yielding  dimension-free estimates for $\sup_{t>t_{G}}|\mathcal M_t^Gf|$ as long as the corresponding dimension-free bounds are available for their continuous analogues. As a corollary, for $q$-balls $G=B^q$, if $p\in(1, \infty]$ and $q\in[1, \infty]$, we obtain that there is a constant $C_{p,q}>0$ independent of the dimension $d\in\NN$ such that  for all $f\in\ell^p(\ZZ^d)$ we have
\begin{align}
\label{eq:1}
\big\|\sup_{t\ge d^{1+1/q}}|\mathcal M_t^{B^q} f|\big\|_{\ell^p(\ZZ^d)}
\leq C_{p,q} \|f\|_{\ell^p(\ZZ^d)}.
\end{align}

At this stage, the whole difficulty lies in estimating $\sup_{0<t\le t_{G}}|\mathcal M_t^Gf|$, where the things are
getting more complicated. Nevertheless, as we shall see below, in some cases  improvements are possible.

We show that in the case of $\mathcal M_t^{B^2}$,  which together with $\mathcal M_t^{B^{\infty}}$, is presumably  the most natural setting for the discrete Hardy--Littlewood maximal functions, the range in \eqref{eq:1} can be improved. Namely, the main discrete result of this paper is, an extension of \eqref{eq:1} for $G=B^2$, stated below. 
\begin{theorem}
	\label{thm:3}
	For each $p\in(1,\infty]$ there is a constant $C_p>0$ independent of the dimension 
        $d\in\NN$ such that for every $f\in\ell^p(\ZZ^d)$ we have
        \begin{align}
          \label{eq:7}
          \big\|\sup_{t\ge Cd}|\mathcal M_t^{B^2}f|\big\|_{\ell^p(\ZZ^d)}\le C_p\|f\|_{\ell^p(\ZZ^d)},
        \end{align}
        for an appropriate absolute constant $C>0$.
\end{theorem}
The proof of Theorem \ref{thm:3} is based on a delicate refinement of the arguments from \cite{BMSW3}, which in the end reduce the matters to the comparison of the norm of $\sup_{t\ge Cd}|\mathcal M_t^{B^2}f|$ with the norm of its continuous analogue, and consequently to the dimension-free estimates of $C_p(d, B^2)$ for all $p\in(1, \infty]$, that are guaranteed by \cite{SteinMax}. The proof of Theorem \ref{thm:3} is contained in Section \ref{sec:6}.

Surprisingly, as it was shown in \cite{BMSW3}, the dimension-free estimates in the discrete case are not as broad as in the continuous setup and there is no obvious conjecture to prove. This is due to the fact that there exists a simple example of a  convex symmetric body in $\ZZ^d$ for which maximal estimate \eqref{eq:86} on $\ell^p(\ZZ^d)$, for every $p\in(1, \infty)$, involves the smallest constant $\mathcal C_{p}(d, G)>0$ unbounded in $d\in\NN$. In order to carry out the construction it suffices to  fix a sequence  $1\leq \lambda_1<\ldots<\lambda_d<\ldots<\sqrt{2}$ and consider, as in \cite{BMSW3}, the ellipsoid 
\begin{align*}
E(d)=\Big\{x\in \RR^d\colon \sum_{k=1}^d \lambda_k^2x_k^2\,\le 1 \Big\}.
\end{align*}
Then one can prove that for every $p\in(1, \infty)$ there is $C_p>0$ such that for every $d\in\NN$ one has
\begin{align}
\label{eq:91}
\mathcal C_{p}(d, E(d))\ge C_p(\log d)^{1/p}.
\end{align}

Inequality \eqref{eq:91} shows that the dimension-free phenomenon  for the Hardy--Littlewood maximal functions in the discrete setting is much more delicate, and the dimension-free estimates even in the Euclidean case for $\mathcal C_p(d, B^2)$ may be very difficult. However, there is an evidence, gained recently by the authors in \cite{BMSW2}, in favor of the general problem, which makes the things not entirely hopeless.
Namely, in \cite{BMSW2} a dyadic variant of inequality \eqref{eq:86} for $G=B^2$ was studied and we proved the following result. 

\begin{theorem}
\label{thm:0}
For every $p\in[2, \infty]$ there exists a constant $C_p>0$ independent of $d\in\NN$ such that for every 
$f\in\ell^p(\ZZ^d)$  we have
\begin{align}
\label{eq:90}
   \big\|\sup_{t\in\mathbb D}|\mathcal M_t^{B^2}f|\big\|_{\ell^p(\ZZ^d)}\le C_p\|f\|_{\ell^p(\ZZ^d)}.
\end{align}
\end{theorem}
 All the aforementioned results give us strong motivation to understand the situation more generally. In particular,  in the case of $q$-balls $G=B^q$ where $q\in[1, \infty)$, which is well understood in the continuous setup. More about the methods available in the discrete setting is in Section \ref{sec:3}.

\subsection{Notation} Here we fix some further notation and terminology. 
\begin{enumerate}[label*={\arabic*}.]

\item Throughout the whole paper $d\in\NN$  denotes the dimension and
$C>0$ denotes a universal constant, which does not depend on the
dimension, but it may vary from occurrence to occurrence.
\item We write that $A \lesssim_{\delta} B$
($A \gtrsim_{\delta} B$) to say that there is an absolute constant
$C_{\delta}>0$ (which possibly depends on $\delta>0$) such that
$A\le C_{\delta}B$ ($A\ge C_{\delta}B$), and  we  write
$A \simeq_{\delta} B$ when $A \lesssim_{\delta} B$ and
$A\gtrsim_{\delta} B$ hold simultaneously.

\item 
Let $\NN=\{1,2,\ldots\}$ be the set of positive integers let $\NN_0 = \NN\cup\{0\}$, and
let $\mathbb D=\{2^n: n\in\NN_0\}$ denote the set of all dyadic numbers.
We set $\NN_N = \{1, 2, \ldots, N\}$ for any $N \in \NN$.

\item 
The Euclidean space $\RR^d$
is endowed with the standard inner product
\[
x\cdot\xi=\langle x, \xi\rangle=\sum_{k=1}^dx_k\xi_k
  \]
for every $x=(x_1,\ldots, x_d)$ and $\xi=(\xi_1, \ldots,
\xi_d)\in\RR^d$.

\item For a countable set $\mathcal Z$ (usually $\mathcal Z=\ZZ^d$) endowed with the counting measure we shall
write that
\[
\ell^p(\mathcal Z)=\{f:\mathcal Z\to \CC: \|f\|_{\ell^p(\mathcal Z)}<\infty\} \quad \text{for any}\quad p\in[1, \infty],
\]
where for any $p\in[1, \infty)$ we have
\begin{align*}                    
  \|f\|_{\ell^p(\mathcal Z)}=\Big(\sum_{m\in\mathcal Z}|f(m)|^p\Big)^{1/p} \qquad \text{and} \qquad
  \|f\|_{\ell^{\infty}(\mathcal Z)}=\sup_{m\in\mathcal Z}|f(m)|.
\end{align*}

\item We shall abbreviate $\|\cdot\|_{L^p(\RR^d)}$ to $\|\cdot\|_{L^p}$, and  $\|\cdot\|_{\ell^p(\ZZ^d)}$ to $\|\cdot\|_{\ell^p}$.

\item Let $(X, \mathcal B, \mu)$ be a $\sigma$-finite measure space. Let $p\in[1, \infty]$ and suppose that
$(T_t)_{t\in\mathbb I}$ is a family of linear operators such that $T_t$ maps $L^p(X)$ to itself for every $t\in \mathbb I\subseteq (0, \infty)$. Then the corresponding maximal function will be denoted by
\[
T_{*, \mathbb I}f=\sup_{t\in\mathbb I}|T_tf|, \quad \text{for every}\quad f\in L^p(X).
\]
We shall abbreviate $T_{*, \mathbb I}$ to $T_{*}$, if $\mathbb I=(0, \infty)$.
\item Let $(B_1, \|\cdot\|_{B_1})$ and $(B_2, \|\cdot\|_{B_2})$ be Banach spaces. For a linear or sub-linear operator $T: B_1\to B_2$ its norm is defined by
\[
\|T\|_{B_1\to B_2}=\sup_{\|f\|_{B_1}\le1}\|T(f)\|_{B_2}.
\]

\item Let $\mathcal F$ denote the Fourier transform on $\RR^d$ defined for any function 
$f \in L^1\big(\RR^d\big)$ as
\begin{align*}
\calF f(\xi) = \int_{\RR^d} f(x) e^{2\pi i \sprod{\xi}{x}} {\: \rm d}x \quad \text{for any}\quad \xi\in\RR^d.
\end{align*}
If $f \in \ell^1\big(\ZZ^d\big)$ we define the discrete Fourier
transform by setting
\begin{align*}
\hat{f}(\xi) = \sum_{x \in \ZZ^d} f(x) e^{2\pi i \sprod{\xi}{x}} \quad \text{for any}\quad \xi\in\TT^d,
\end{align*}
where $\TT^d\equiv [-1/2, 1/2)^d$ is the $d$-dimensional torus. We shall denote by $\mathcal F^{-1}$ the inverse Fourier transform on $\RR^d$
or the inverse Fourier transform (Fourier coefficient) on the torus $\TT^d$.
This will cause no confusions and the meaning will be always clear from the
context. 
\end{enumerate}

\section*{Acknowledgements}
The authors are grateful to the referees for careful reading of the manuscript and useful remarks
that led to the improvement of the presentation.

\section{A review of the current state of the art}
\label{sec:2}

In the 1980s dimension-free estimates for the Hardy--Littlewood maximal functions over convex symmetric  bodies had begun to be studied \cite{SteinMax, SteinStro} and went through a period of considerable changes and developments \cite{B33, B1, B2, Car1, Mul1}.  However, the dimension-free phenomenon in harmonic analysis had been apparent much earlier, see for instance \cite[Chapter 14, \S 3 in Vol.II]{Zyg}, as well as \cite{Ste1} and the references given there. 
We refer also to more recent results \cite{Ald1, Aub1, B3, BMSW1, BMSW3, BMSW2, IakStr1, MSZ1, Som} and the survey article \cite{DGM1} for a very careful and detailed exposition of the subject.

\subsection{Dimension-free estimates for semigroups}
  Consider the Poisson semigroup $(P_t)_{t\ge0}$ defined  on the Fourier transform side by
  \begin{align*}
  \calF (P_t f)(\xi)=p_t(\xi)\calF (f)(\xi),
  \end{align*}
  for every $t\ge0$ and  $\xi\in\RR^d$, with the symbol 
  \[
p_t(\xi)=e^{-2\pi tL|\xi|},
\]
involving an isotropic constant $L=L(G)>0$ defined in \eqref{eq:iso}. The dilation by the isotropic constant  is a technical assumption, which will simplify our further discussion.

For every $x\in\RR^d$ we introduce the maximal function 
\begin{align*}
P_{*}f(x)=\sup_{t>0}|P_tf(x)|,
\end{align*}
and the square function
\begin{align*}
g(f)(x)=\left(\int_{0}^{\infty}t\Big|\frac{\rm d}{{\rm d}t} P_t
  f(x)\Big|^2{\rm d}t\right)^{1/2},
\end{align*}
associated with the Poisson semigroup. From  \cite{Ste1} we know that for every
$p\in(1, \infty)$ there exists a constant $C_p>0$, which does not depend on
$d\in\NN$, such that for every $f\in L^p(\RR^d)$ we have
\begin{align}
  \label{eq:47'}
\|P_{*}f\|_{L^p}\leq C_p\|f\|_{L^p},
\end{align}
and
\begin{align}
  \label{eq:48}
  \|g(f)\|_{L^p}\leq C_p\|f\|_{L^p}.
  \end{align}
For the proof of \eqref{eq:47'} and \eqref{eq:48} one has to check that $(P_t)_{t\ge0}$ is a symmetric diffusion semigroup in the sense of \cite[Chapter III]{Ste1}. For the convenience of the reader we recall the definition of a symmetric diffusion semigroup from \cite[Chapter III, p.65]{Ste1}.  Let $(X, \mathcal B(X), \mu)$ be a $\sigma$-finite measure space, and $(T_t)_{t\ge0}$ be a strongly continuous semigroup on $L^2(X)$, which maps $L^1(X)+L^{\infty}(X)$ to itself for every $t\ge0$. We say that $(T_t)_{t\ge0}$ is a symmetric diffusion semigroup, if it satisfies for all $t\ge0$ the following conditions: 
\begin{enumerate}[label*={\arabic*}.]
\item[1.] {\it Contraction property:} for all $p\in[1, \infty]$ and $f\in L^p(X)$ we have
$\|T_tf\|_{L^p(X)}\le \|f\|_{L^p(X)}$.
\item[2.] {\it Symmetry property:} each $T_t$ is a self-adjoint operator on $L^2(X)$.
\item[3.] {\it Positivity property:} $T_tf\ge0$, if $f\ge0$.
\item[4.] {\it Conservation property:} $T_t1=1$. 
\end{enumerate}

One major advantage  of using the above-mentioned conditions is that the probabilistic techniques are applicable to understand properties of $T_t$. This is the reason why, in particular, inequalities \eqref{eq:47'} and \eqref{eq:48} hold, see \cite[Chapter III]{Ste1} for more details, and also \cite{Cow} for an even more relaxed conditions.
The semigroup $P_t$ is closely linked to the averaging operator $M_t^G$. Namely, both operators are contractive on  $L^p(\RR^d)$ spaces for all $p\in[1, \infty]$, preserve the class of nonnegative functions, and satisfy $P_t1=M_t^G 1=1$.

Later on, we shall need a  variant of the  Littlewood--Paley inequality. For every $n\in\ZZ$ we define the Poisson projections
$S_n$ by setting
\begin{align*}
S_n=P_{2^{n-1}}-P_{2^n}.  
\end{align*}
 Then, the sequence $(S_n)_{n\in\ZZ}$ is a resolution of the identity on $L^2(\RR^d)$. Namely, we have
 \begin{align}
\label{eq:5}
f=\sum_{n\in \ZZ} S_n f,\quad \text{for every} \quad f\in L^2(\RR^d).   
 \end{align}
Observe that
 \begin{align*}
 S_nf(x)=-\int_{2^{n-1}}^{2^n}\frac{\rm d}{{\rm d}t} P_t f(x){\rm d}t.
 \end{align*}
 Then by the Cauchy--Schwarz inequality we obtain, for every
 $n\in\ZZ$ and $x\in\RR^d$, the following bound
\begin{align*}
|S_nf(x)|^2\leq \bigg(\int_{2^{n-1}}^{2^n}\Big|\frac{\rm d}{{\rm d}t} P_t
  f(x)\Big|{\rm d}t\bigg)^2
  \le 2^{n-1}\int_{2^{n-1}}^{2^n}\Big|\frac{\rm d}{{\rm d}t} P_t
  f(x)\Big|^2{\rm d}t
  \le \int_{2^{n-1}}^{2^n}t\Big|\frac{\rm d}{{\rm d}t} P_t
  f(x)\Big|^2{\rm d}t.
\end{align*}
Now summing over $n\in\ZZ$ and using \eqref{eq:48} one shows that for every $p\in(1, \infty)$, there is a constant $C_p>0$ independent of  $d\in\NN$ such that for every $f\in L^p(\RR^d)$  the following  Littlewood--Paley inequality holds
 \begin{align}
\label{eq:6}
   \Big\|\big(\sum_{n\in \ZZ}|S_n f|^2\big)^{1/2}\Big\|_{L^p}\leq C_p\|f\|_{L^p}.
 \end{align}
Inequality \eqref{eq:6} will play an important role in the proof of Theorem \ref{thm:4}.

We finish this subsection by showing a simple pointwise inequality between the Poisson semigroup and the  Hardy--Littlewood maximal function associated with the Euclidean balls,  which motivates, to some extent, the study of dimension-free estimates for the Hardy--Littlewood maximal functions. Namely, let $K_t$ be the kernel corresponding to $P_t$, assume that $f\ge0$ and observe that
\begin{align*}
P_tf(x)=K_t*f(x)
=
\int_{\mathbb{R}^d} \int_0^{K_t(x-y)} {\rm d}s f(y){\rm d}y
= 
\int_0^{\infty} \int_{ \{y\in\RR^d:K_t (x-y)\ge s \} } f(y) {\rm d}y {\rm d}s.
\end{align*}
The set $\{y\in\RR^d:K_t(x-y)\ge s\}$ is an Euclidean ball centered at $x\in\RR^d$, since $K_1$ is radially decreasing. Thus
\begin{align*}
P_tf(x)=K_t*f(x)
\le 
\bigg(\int_0^{\infty} |\{y\in \RR^d:K_t(x-y)\ge s\}|{\rm d}s\bigg)M_*^{B^2}f(x)
=\|K_t\|_{L^1} M_*^{B^2}f(x).
\end{align*}
Hence we conclude that
\begin{align}
\label{eq:21}
P_* f(x)\le M^{B^2}_{*} f(x). 
\end{align}
Inequality \eqref{eq:47'} gives us a bound independent of the dimension for $\|P_*\|_{L^p\to L^p}$, and in view of \eqref{eq:21} we obtain $\|P_*\|_{L^p\to L^p}\le C_p(d, B^2)$. Now a natural question arises whether  $C_p(d, B^2)$ can be bounded independently of the dimension. This problem was investigated by the third author in \cite{SteinMax}.

\subsection{The case of the Euclidean balls \cite{SteinMax,SteinStro}}
\label{sec:Ste}
The third author obtained in \cite{SteinMax}, see also the joint paper with Str\"omberg \cite{SteinStro} for more details, that for every $p\in(1, \infty]$ there is a constant $C_p>0$ independent of the dimension $d\in\NN$ such that
\begin{align}
\label{eq:36}
C_p(d, B^2)\le C_p.
\end{align}
Let us briefly describe the method used in \cite{SteinMax} to prove \eqref{eq:36}. In $\RR^d$, as $d\to\infty$, most of the mass of the unit ball $B^2$ concentrates at the unit sphere $\mathbb S^{d-1}$ in $\RR^d$. In fact, if $\varepsilon\in(0, 1)$, we have
\begin{align*}
d\int_{0}^{1}r^{d-1}{\rm d}r=1,\qquad \text{while}\qquad \lim_{d\to\infty}d\int_{0}^{1-\varepsilon}r^{d-1}{\rm d}r=0.
\end{align*}
Therefore, the key idea is to use the spherical averaging operator, defined for any $r>0$ and $x\in\RR^d$ by
\begin{align}
\label{eq:42}
{A_r^d} f(x)=\int_{\mathbb S^{d-1}} f(x-r\theta){\rm d}\sigma_{d-1}(\theta),
\end{align}
where $\sigma_{d-1} $ denotes the normalized surface measure on $\mathbb S^{d-1}$. Using polar coordinates one easily sees that
\begin{align*}
M_{t}^{B^2}f(x)=d\int_{0}^{1}r^{d-1}{A_{tr}^d} f(x){\rm d}r,
\end{align*}
which immediately implies
\begin{align}
\label{eq:46}
|M_{*}^{B^2}f(x)|\le |A_{*}^df(x)|.
\end{align}
By the earlier result of the third author \cite{Ste0}, we know that for every $d\ge3$ and for every $p>\frac{d}{d-1}$ there is a constant $C_{d, p}>0$ such that for every $f\in L^p(\RR^d)$ one has
\begin{align}
\label{eq:41}
\|A_{*}^df\|_{L^p}\le C_{d, p} \|f\|_{L^p}.
\end{align}
 Inequality \eqref{eq:41} is also true when $d=2$, but this turned out to be a more difficult result, obtained by the first author in \cite{B0}.
Now, the matters are reduced to show that the best constant in \eqref{eq:41} can be taken to be independent of the dimension. For this purpose, the method of rotations enables one to view high-dimensional spheres as an average of rotated low-dimensional ones, and consequently one can conclude that for every $d\ge3$ and $p>\frac{d}{d-1}$ we have
\begin{align}
\label{eq:45}
\|A_{*}^{d+1}\|_{L^p(\RR^{d+1})\to L^p(\RR^{d+1})}\le \|A_{*}^{d}\|_{L^p(\RR^{d})\to L^p(\RR^{d})}.
\end{align}
Hence the best constant in \eqref{eq:41} is non-increasing, and in particular bounded, in $d\in\NN$.

In order to prove \eqref{eq:36} it suffices to take an integer $d_0>\frac{p}{p-1}$. If $d\le d_0$, then there is nothing to do. If $d>d_0$, taking into account \eqref{eq:46} and \eqref{eq:45}, we see that
\begin{align*}
\|M_{*}^{B^2}f\|_{L^p}
\le \|A_{*}^{d_0}\|_{L^p(\RR^{d_0})\to L^p(\RR^{d_0})}\|f\|_{L^p},
\end{align*}
and we obtain \eqref{eq:36} as claimed.

The method described above is limited to the Euclidean balls. The case of general convex symmetric bodies will require a different approach.   

\subsection{The $L^2$ result for general symmetric bodies via Fourier transform methods \cite{B1} }

In \cite{B1} the first author proposed a different approach, which is based on the estimates of the averaging operators $M_t^G$ on the  Fourier transform side. Before we present the main result from \cite{B1} we have to fix some notation and terminology. We begin with the most important definition of this paper.
\begin{definition}
\label{def:1}
We say that a convex symmetric body $G\subset\RR^d$ is in the {\it isotropic position}, if it has Lebesgue measure $|G|=
1$, and there exists a constant $L=L(G)>0$ depending only on $G$ such that
\begin{equation}
\label{eq:iso}
\int_{G}\langle x, \xi\rangle^2{\rm d}x=L(G)^2|\xi|^2\quad\text{for any}\quad \xi\in \RR^d.
\end{equation}
 The constant $L(G)$ in \eqref{eq:iso}  is called the {\it isotropic constant} of $G$.
\end{definition}
From \eqref{eq:iso} one can deduce the following expression for the isotropic constant
\begin{equation}
\label{eq:iso'}
L(G)^2=\frac{1}{d} \int_{G} |x|^2{\rm d}x.
\end{equation}
\begin{lemma}
\label{lem:7}
For every convex symmetric body $G\subset\RR^d$, there exists a linear transformation $U$ of $\RR^d$ such that $U(G)$ is in the isotropic position.
\end{lemma}
\begin{proof}
Observe that
\begin{align*}
M(\xi)=\int_{G}\langle x, \xi\rangle x{\rm d}x=\bigg(\int_{G}\langle x, \xi\rangle x_1{\rm d}x,\ldots,\int_{G} \langle x, \xi\rangle x_d{\rm d}x\bigg)
\end{align*}
is a positive operator on $\RR^d$. Thus one can find a positive operator $S$ such that $M=S^2$.  Setting $U=c(G, S)S^{-1}$, where $c(G, S)=|\det S|^{1/d}|G|^{-1/d}$, we see that $|U(G)|=1$ and 
\begin{align*}
\int_{U(G)}\langle x, \xi\rangle^2{\rm d}x&=c(G, S)^2|G|^{-1}\int_{G}\langle S^{-1}x, \xi\rangle^2{\rm d}x\\
&=c(G, S)^2|G|^{-1}\langle M(S^{-1}\xi),S^{-1}\xi \rangle\\
&=c(G, S)^2|G|^{-1}|\xi|^2.
\end{align*}
Hence $U(G)$ is in the isotropic position, with the isotropic constant $L(U(G))=c(G, S)|G|^{-1/2}>0$.
\end{proof}

Observe that if the body $G$ in \eqref{eq:109} is replaced with any other set of the form $U(G)$, where $U$ is an invertible linear transformation of $\RR^d$, then the $L^p(\RR^d)$ bounds from \eqref{eq:109} remain unchanged and we have
\begin{align}
\label{eq:47}
C_p(d, G)=C_p(d, U(G)).
\end{align}
Indeed, considering an isometry $U_{p}$ of $L^p(\RR^d)$ given by
\[
U_{p}f=|\det U|^{-1/p}f\circ U^{-1}, \quad \text{for any}\quad  p\ge1,
\]
we obtain \eqref{eq:47}, since 
\begin{align*}
U_{p} \circ M_t^{G}=M_t^{U(G)} \circ U_{p}.
\end{align*}

In view of \eqref{eq:47} the dimension-free estimates are unaffected by a change of the underlying body to an equivalent one. Therefore, from now on unless otherwise stated,  we assume that $G\subset\RR^d$ is in the isotropic position. 
For a symmetric convex body $G\subset\RR^d$, let
\begin{align*}
m^G(t\xi)=\mathcal F (\ind{G})(t\xi)
\end{align*}
be the multiplier corresponding to the operator $M_t^G$ from \eqref{eq:93}.

In \cite{B1} the first author provided the  estimates for $m^G$ and its derivatives in terms of the isotropic constant $L(G)$, see Theorem \ref{prop:1} below. 
\begin{theorem}[{\cite[eq. (10),(11),(12)]{B1}}]
\label{prop:1}
Let $G$ be   a symmetric convex body  $G\subset \RR^d$ which is in the isotropic position. Let  $L=L(G)$ be the isotropic constant of
$G$. Then for every $\xi\in\RR^d\setminus\{0\}$ we have
\begin{align}
\label{eq:69}
|m^G(\xi)|\leq 150(L |\xi|)^{-1},\qquad\text{and}\qquad  |m^G(\xi)-1|\leq 150(L|\xi|),
\end{align}
and
\begin{equation}
\label{eq:69'}
|\langle\xi,\nabla m^G(\xi)\rangle|\le 150.
\end{equation}
\end{theorem}
In Section \ref{sec:5}, for the sake of completeness,  we provide a detailed proof of Theorem \ref{prop:1}. In fact, as we shall see later on, the estimates in \eqref{eq:69} and \eqref{eq:69'} will be the core of the proof of Theorem \ref{thm:4}.

Using Theorem \ref{prop:1}, as the main tool, it was proved in \cite{B1} that 
\begin{align}
\label{eq:55}
C_2(d, G)\le C,
\end{align}
where $C>0$ is a constant that does not depend neither
on $d\in\NN$ nor the underlying body $G\subset \RR^d$.
In view of the dimensional-free estimates for the Poisson semigroup \eqref{eq:47'} in order to prove \eqref{eq:55} it suffices to obtain the following dimensional-free maximal estimate
\begin{align}
\label{eq:56}
\|\sup_{t>0} |(M_{t}^G-P_t)f|\|_{L^2}\le C\|f\|_{L^2}.
\end{align}
The estimate \eqref{eq:56}, in turn, was reduced, using some square function argument and the Plancherel theorem, to the uniform in $\xi\in \RR^d $ estimate  
\begin{equation}
\label{eq:57}
\sum_{n\in \ZZ}\min\big\{(2^n L(G)|\xi|),(2^n L(G)|\xi|)^{-1}\big\}\le C,
\end{equation}
where $C>0$ is a universal constant independent of $d\in\NN$ and the body $G\subset\RR^d$. It is easy to see that \eqref{eq:57} indeed holds. Moreover, it is true regardless of the exact value of the isotropic constant $L(G)$. Remarkably, we do not need to know whether $L(G)$ is comparable to a dimension-free constant.

\subsection{Interlude: the isotropic conjecture}
As we have already underlined, the approach from \cite{B1} does not require any information on the size of the isotropic constant $L(G).$
Recall at this point that $L(G)$
is known to be bounded from below by an absolute constant. 
\begin{proposition}
\label{prop:isobel}
There is a universal constant $c>0$ independent of the dimension such that for all convex symmetric bodies $G\subset\RR^d$ in the isotropic position we have $L(G)\ge c$.
\end{proposition}
\begin{proof}
Let $r_d$ be such that $|r_d B^2|=1$. Then $r_d B^2$ is in the isotropic position and $r_d=|B^2|^{-1/d}\simeq d^{1/2}$.
Using \eqref{eq:iso'} and polar coordinates one has
\begin{align*}
L(r_d B^2)^2=\frac{1}{d}\int_{ r_d B^2}|x|^2{\rm d}x=\frac{|B^2|r_d^{d+2}}{d+2}=\frac{r_d^{2}}{d+2}\simeq 1.
\end{align*}
Clearly, $|x|\ge r_d$ on $G\setminus r_d B^2$ and $|x|\le r_d$ on $r_dB^2\setminus G.$ Thus, using \eqref{eq:iso'} and the observation that $G\setminus r_d B^2$ and $r_d B^2\setminus G$ have the same volume, we estimate
\begin{align*}
dL(G)^2&=\int_{G} |x|^2{\rm d}x= \int_{G\cap r_d B^2} |x|^2{\rm d}x
+\int_{G\setminus r_d B^2} |x|^2{\rm d}x
\ge\int_{G\cap r_d B^2} |x|^2{\rm d}x
+ r_d^2|G\setminus r_d B^2|\\
&\ge \int_{G\cap r_d B^2} |x|^2{\rm d}x
+\int_{r_d B^2\setminus G} |x|^2{\rm d}x=
dL(r_d B^2)^2.
\end{align*}
Therefore, we see that $L(G)\ge L(r_d B^2)\ge c>0$. This completes the proof.  
\end{proof}

Conversely, it is not difficult to show the following upper bound. 
\begin{proposition}
\label{prop:isoup}
There is a universal constant $C>0$ independent of the dimension such that for all convex symmetric bodies $G\subset\RR^d$ in the isotropic position we have $L(G)\le Cd^{1/2}$.
\end{proposition}
\begin{proof}
If $r(G)$  is the largest radius $r>0$ such that $rB^2\subseteq G$ then there is an absolute constant $c>0$ such that
\begin{align*}
cL(G)\le r(G),
\end{align*}
we refer to \cite[Section 3.1, p.108]{BGVV} for more details. 
It follows that $cL(G)B^2\subseteq G$ and 
\begin{align*}
(cL(G))^d|B^2|\le |G|=1,
\end{align*}
and consequently, using  $|B^2|^{-1/d}\simeq d^{1/2}$ we obtain the desired claim. 
\end{proof}

The estimate from Proposition \ref{prop:isoup} was improved by the first author in \cite{B00}, where it was shown that $L(G)=O(d^{1/4}\log d)$. Klartag \cite{Kla1} proved that $L(G)=O(d^{1/4})$, and this is the best currently available general estimate for $L(G)$. However, the uniform bound from above for $L(G)$ is a well-known open problem with several equivalent formulations. More precisely, we are  lead to the following conjecture.
\begin{conjecture}
\label{con:iso}
There is a constant $C>0$ independent of $d\in\NN$ such that for all convex symmetric bodies $G\subset\RR^d$ in the isotropic position we have $L(G)\le C$.
\end{conjecture}

This conjecture was verified for various classes of convex symmetric bodies. To give an example we consider the class of $1$-unconditional symmetric convex bodies. Let $\{e_1,\ldots, e_d\}$ denote the canonical basis in  $\RR^d$. We say that $G\subset\RR^d$ is such a body, whenever, for every choice of signs $\varepsilon_1,\ldots,\varepsilon_d\in\{-1,1\}$, we have
\begin{align*}
\Big\|\sum_{i=1}^d \varepsilon_i x_i e_i\Big\|_G=\|x\|_G,\quad \text{for all}\quad x=(x_1,\ldots, x_d)\in\RR^d,
\end{align*}
where $\|x\|_G=\inf\{t>0\colon x\in tG\}$ denotes the Minkowski norm  associated with $G$.
\begin{proposition}
\label{prop:1unciso}
There is a constant $C>0$ independent of $d\in\NN$ such that for all $1$-unconditional  convex bodies $G\subset\RR^d$ in the isotropic position
we have $L(G)\le C$.
\end{proposition}
For the proof of Proposition \ref{prop:1unciso}, and a more detailed exposition about the subject of geometry of isotropic convex bodies  we refer to the monograph \cite{BGVV}. Interestingly, the issue of the isotropic constant did not impact the proofs of the dimension-free
bounds for the Hardy--Littlewood maximal function \eqref{eq:109} obtained in \cite{B1}. This gives us strong motivation to understand the role of the isotropic constant $L(G)$ in the estimates for $C_p(d, G)$ for all $p\in(1, \infty]$ for  general convex symmetric bodies $G\subset\RR^d$.

\subsection{The $L^p$ results for  $p\in(3/2, \infty]$ and fractional integration method}
The first author \cite{B2}, and independently Carbery \cite{Car1}, extended the $L^2(\RR^d)$ result from \cite{B1}, and showed that for every $p\in(3/2, \infty]$ there exists a numerical constant $C_p>0$, which does not depend on the dimension $d\in\NN$ such that for every convex symmetric body $G\subset\RR^d$ we have
\begin{align}
\label{eq:54}
C_p(d, G)\le C_p.
\end{align}

They also showed that if the supremum in \eqref{eq:109} is restricted to the set of dyadic numbers $\mathbb D$, then inequality \eqref{eq:54} remains valid for all $p\in(1, \infty]$. The methods used in these papers were completely different. We shall focus our attention merely on Carbery's paper \cite{Car1}, since it was an important starting point for the papers  \cite{Mul1} and \cite{B3}, which will be discussed in the next subsection. 
   
The first main idea introduced in \cite{Car1} reduced inequality \eqref{eq:54} to proving that for every $p\in(3/2, 2]$ there exists $C_p>0$ independent of  $d\in\NN$ such that for every convex symmetric body $G\subset\RR^d$ we have 
\begin{align}
\label{eq:58}
\sup_{n\in\ZZ}\big\|\sup_{t \in [2^n,2^{n+1}]}|M_t^G f|\big\|_{L^p} \le C_p  \|f\|_{L^p},
\end{align}
for every $f\in L^p(\RR^d)$. This was achieved by appealing to an almost orthogonality principle, which combined with the Littlewood--Paley inequality \eqref{eq:6} and inequality \eqref{eq:47'} for the Poisson semigroup, resulted in \eqref{eq:54}. The author of \cite{Car1} adjusted an almost orthogonality principle to dimension-free setting from the unpublished notes of Christ, see also \cite{Car2} for a more detailed discussion.

The second main idea of \cite{Car1} relies on a fractional derivative/integration method, and it was  used to prove \eqref{eq:58}. Let $\calF_{\RR}$ denote the one dimensional Fourier transform. For $\alpha\in(0, 1)$, let $\mathcal D^{\alpha}$ be the fractional derivative
\begin{equation*}
%\label{eq:104}
\mathcal D^{\alpha}F(t)=\mathcal D^{\alpha}_tF(t)=\mathcal D^{\alpha}_uF(u)\big|_{u=t}
=\calF_{\RR}\big((2i\pi \xi )^{\alpha}\calF_{\RR}^{-1}(F)(\xi)\big)(t),\quad  \text{for}\quad t\in\RR.
\end{equation*}
This formula gives a well defined tempered distribution on $\RR$. Simple computations show that 
\begin{align*}
\mathcal D_t^{\alpha}m^G(t\xi)=\int_{G}(2\pi i
x\cdot\xi)^{\alpha}e^{-2\pi i tx\cdot\xi}{\rm d}x,\quad  \text{for}\quad t>0,
\end{align*}
where $m^G(t\xi)=\mathcal F(\ind{G})(t\xi)$.
Moreover, \cite[Lemma 6.6]{DGM1} guarantees that
\begin{align*}
\mathcal D_t^{\alpha}m^G(t\xi)
=-\frac1{\Gamma(1-\alpha)}\int_t^{\infty}(u-t)^{-\alpha}\frac{{\rm d}}{{\rm d}u}m^G(u\xi){\rm d}u, \quad \text{for}\quad \xi\in\RR^d.
\end{align*}
If $\mathcal P_u^{\alpha}$ is the operator associated with the multiplier
\begin{align*}
\mathfrak p_u^{\alpha}(\xi)=u^{\alpha+1}\mathcal D_v^{\alpha}\bigg(\frac{m^G(v\xi)}{v}\bigg)\bigg|_{v=u}, \quad \text{for}\quad \xi\in\RR^d,
\end{align*}
then one can see that
\begin{align}
\label{eq:99}
M_t^Gf(x)=\mathcal F^{-1}(m^G(t\xi)\mathcal F
f)(x)=\frac{1}{\Gamma(\alpha)}\int_t^{\infty}\frac{t}{u}
\bigg(1-\frac{t}{u}\bigg)^{\alpha-1}\mathcal P_{u}^{\alpha}f(x)\frac{{\rm d}u}{u}.
\end{align}
It was shown \cite{Car1} that for general symmetric convex bodies one has
\begin{align}
\label{eq:34'}
\|\mathcal P_1^{\alpha} f\|_{L^p}\lesssim \|f\|_{L^p}+\|T_{(\xi\cdot\nabla)^{\alpha}m^G} f\|_{L^p},
\end{align}
where $T_{(\xi\cdot\nabla)^{\alpha}m^G} f$ is  the multiplier operator associated with the symbol
\[
(\xi\cdot\nabla)^{\alpha}m^G(\xi)=\mathcal D_t^{\alpha}m^G(t\xi)|_{t=1}.
\]
The estimate from \eqref{eq:34'} immediately implies that
\begin{align}
\label{eq:38'}
\sup_{u>0}\|\mathcal P_u^{\alpha}\|_{L^p\to L^p}
\lesssim_p 1+\|T_{(\xi\cdot\nabla)^{\alpha}m^G}\|_{L^p\to L^p},
\end{align}
since the
multipliers $\mathfrak p_u^{\alpha}$ are dilations of $\mathfrak p_1^{\alpha}$.
Using  \eqref{eq:99} and \eqref{eq:38'}  one controls
\begin{equation}
\label{eq:fracest}
\big\|\sup_{t \in [2^n,2^{n+1})}|M^G_t f|\big\|_{L^p}
\le C_p  \big(1+\|T_{(\xi\cdot\nabla)^{\alpha}m^G}\|_{L^p\to L^p}\big)\|f\|_{L^p},
\end{equation} 
whenever $\alpha>1/p$. Now, since $T_{(\xi\cdot\nabla)^{1}m^G}$ is associated with the symbol $(\xi\cdot \nabla) m^G=\xi\cdot \nabla m^G(\xi)$, by Plancherel's theorem and \eqref{eq:69'} we have
$\|T_{(\xi\cdot\nabla)^{1}m^G}\|_{L^2\to L^2}\le C$. Clearly, $T_{(\xi\cdot\nabla)^{0}m^G}=M_1$ is a contraction on $L^1(\RR^d)$. Then by complex interpolation, as in \cite{Car1}, we get  $\|T_{(\xi\cdot\nabla)^{\alpha}m^G}\|_{L^p\to L^p}\le C_{\alpha}$, whenever $\alpha<2/p'$. In view of the restriction for $\alpha>1/p$ in \eqref{eq:fracest} we obtain \eqref{eq:58} for $p\in(3/2, 2]$. 

The above-mentioned method of fractional integration was exploited in \cite{Mul1} and \cite{B3}.

\subsection{The  $L^p$ result for $p\in(1, \infty]$,  the case of $q$-balls}
M\"uller \cite{Mul1} proved, for all $p\in[1, \infty]$ and for every symmetric convex body $G\subset\RR^d$,  a remarkable upper bound for $C_p(d, G)$  in terms of certain geometric invariants. To be more precise, assuming that the body $G$ is in the isotropic position, we  define two constants, geometric invariants, by setting 
\begin{align*}
%\label{eq:59}
\sigma(G)^{-1}=\max\big\{\varphi^G_{\xi}(0):\xi\in\mathbb S^{d-1}\big\},
\end{align*}
and
\begin{align*}
%\label{eq:60}
Q(G)&=\max\big\{{\rm Vol}_{d-1}(\pi_{\xi}(G)): \xi\in\mathbb S^{d-1}\big\},
\end{align*}
where $\varphi^G_{\xi}(0)={\rm Vol}_{d-1}(\{x\in G\colon x\cdot \xi =0\})$, while $\pi_{\xi}:\RR^d\to\xi^{\perp}$ denotes the orthogonal projection
of $\RR^d$ onto the hyperplane perpendicular to $\xi$.  It follows from \eqref{eq:fiL} that 
$\frac{1}{8}L(G)\le \sigma(G)\le 8L(G)$.

Using these two linear invariants $\sigma(G)$ and $Q(G)$ it was proved in \cite{Mul1}  that for every $p\in(1, \infty]$ and for every symmetric convex body $G\subset\RR^d$ there is a constant $C(p, \sigma(G), Q(G))>0$ independent of the dimension $d\in\NN$ such that
\begin{align}
\label{eq:100}
 C_p(d, G)\le C(p, \sigma(G), Q(G)).
\end{align}
In other words $C_p(d, G)$ may depend on $\sigma(G)$ and $Q(G)$, but not explicitly on the dimension $d\in\NN$. For $p\in(3/2, \infty]$ inequality \eqref{eq:100}  is weaker than the estimates from \cite{B2} and \cite{Car1}, which show that $C_p(d, G)$ can be even chosen independently of $d$ and $G$. However, using \eqref{eq:100} it was proved that
$C_p(d, B^q)$ is independent of the dimension for all  $q\in[1, \infty)$, since $\sigma(B^q)$ and $Q(B^q)$ can
be explicitly computed and they are independent of the dimension, but they depend on $q$.  For the cubes $G=B^{\infty}$ it turned out that $\sigma(B^{\infty})$ is independent of the dimension, but $Q(B^{\infty})=d^{1/2}$, and at that time the cubes were thought of as  candidates for a counterexample. However, the first author refined M\"uller's approach, and provided the dimensional-free bounds for $C_p(d, B^{\infty})$ for all $p\in(1, \infty]$ as well.
We shall now give a description of M\"uller's methods, which resulted in inequality \eqref{eq:100}.

As in \cite{Car1}, the proof of \eqref{eq:100} in \cite{Mul1} is reduced to estimates of the $L^p(\RR^d)$ norm of the operator $T_{(\xi\cdot\nabla)^{\alpha}m^G}$. Recall, that the complex interpolation allowed Carbery to prove dimension-free $L^p(\RR^d)$ bounds for $T_{(\xi\cdot\nabla)^{\alpha}m^G}$ only in the restricted range of $\alpha<2/p'$. M\"uller, by considering a suitable  admissible family of Fourier multiplier operators, was able to prove that, for all $p\in(1, \infty)$ and for all $\alpha\in(1/2,1)$, one has 
\begin{equation*}
%\label{eq:Mul1}
\|T_{(\xi\cdot\nabla)^{\alpha}m^G}\|_{L^p\to L^p}
\le C_{p}(\alpha,\sigma(G),Q(G)).
\end{equation*}
More precisely, by using complex interpolation it was shown in \cite{Mul1} that
\begin{equation}
\label{eq:Mul2}
\|T_{(\xi\cdot\nabla)^{\alpha}m^G}\|_{L^p\to L^p}
\le C_{\alpha}\big(1+\|T_{-2\pi |\xi| m^G(\xi)}\|_{L^p\to L^p}\big),
\end{equation}
for $\alpha\in(1/2,1)$, where $T_{-2\pi|\xi| m^G(\xi)}$ is the multiplier operator associated with the symbol $-2\pi|\xi| m^G(\xi)$.

Finally, \eqref{eq:Mul2} reduced the task to justifying
\begin{equation}
\label{eq:Mul3}
\|T_{-2\pi|\xi| m^G(\xi)}\|_{L^p\to L^p}\le  C_{p}(\sigma(G),Q(G)),
\end{equation}
for all $p\in(1,\infty)$. Since $T_{-2\pi|\xi| m^G(\xi)}$ is self-adjoint while proving \eqref{eq:Mul3} we can assume that $p\in[2,\infty)$. The key part of the proof of \eqref{eq:Mul3} in \cite{Mul1} is based on the following identity
\begin{align*}
-2\pi|\xi|m^G(\xi)=\sum_{j=1}^d\bigg(-i\frac{\xi_j}{|\xi|}\bigg) (-2\pi i\xi_j m^G(\xi)).
\end{align*}
Thus, defining the measures $\mu_j=\frac{{\rm d}}{{\rm d}x_j} \ind{G}(x)$ we see that $$T_{-2\pi|\xi| m^G(\xi)}=\sum_{j=1}^d R_j(\mu_j *f),$$
where $R_j$ is the Riesz transform, corresponding to the multiplier $-i\xi_j/|\xi|$ for $j\in\{1, \ldots, d\}$.

We now are at the stage, where the dimension-free estimates for the vector of Riesz transforms enter into the game. The third author \cite{SteinRiesz} proved that for every $p\in(1, \infty)$ there is a constant $C_p>0$ independent of the dimension $d\in\NN$  such that the following estimate 
\begin{align}
\label{eq:62}
\Big\|\big(\sum_{j=1}^d |R_j f|^2\big)^{1/2}\Big\|_{L^p}\leq C_p\|f\|_{L^p},
\end{align}
holds for every $f\in L^p(\RR^d)$.

Then, dimension-free estimates for the vector of Riesz transforms \eqref{eq:62} on $L^{p'}(\RR^d)$, together with a duality argument reduce the problem to the following square function estimate
\begin{equation}
\label{eq:Mul4}
\Big\|\big(\sum_{j=1}^d |\mu_j *f|^2\big)^{1/2}\Big\|_{L^p}\le C_{p}(\sigma(G),Q(G)) \|f\|_{L^p},
\end{equation}
for $p\in[2,\infty)$, which was achieved by interpolating between its $p=2$ and $p=\infty$ endpoints.

As it has been mentioned above this approach resulted in dimension-free estimates for $C_p(d, B^q)$ for all $p\in(1, \infty]$ and $q\in[1, \infty)$, since in these cases the geometric invariants   $\sigma(B^q)$ and $Q(B^q)$ turned out to be independent of $d\in\NN$. For $q=\infty$ one obtains $Q(B^{\infty})=d^{1/2}$, which resulted in no further progress for the Hardy--Littlewood maximal function for the cube.

However, for $q=\infty$ the first author observed \cite{B3} by a careful inspection of M\"uller's proof, that \eqref{eq:Mul4} for $p=2$ can be estimated by a constant, which depends only on $\sigma(B^{\infty})$, and the dependence on $Q(B^{\infty})$ enters in  \eqref{eq:Mul4} only for $p=\infty$. Therefore, instead of interpolating between $p=2$ and $p=\infty$ in \eqref{eq:Mul4} it was natural to try, loosely speaking, to bound \eqref{eq:Mul4} for $p=q$ with large $q\ge2$, and then interpolate with the improved estimate for $p=2$, to obtain \eqref{eq:Mul4} with the implied constant depending only on $p$ and  $\sigma(B^{\infty})$.   
In \cite{B3}, in the proof of \eqref{eq:Mul4} for $p=q$ with large $q\ge2$  the explicit formula for the multiplier 
\begin{equation*}
%\label{eq:mcubeform}
m^{B^{\infty}}(\xi)=\prod_{j=1}^d \frac{\sin (\pi \xi_j)}{\pi\xi_j},\quad \text{for}\quad \xi \in \RR^d
\end{equation*}
was essential. From Theorem \ref{prop:1} we have seen that $|m^{B^{\infty}}(\xi)|\le C|\xi|^{-1}$. However,  $m^{B^{\infty}}(\xi)$, for most of $\xi$, decays much faster than $|\xi|^{-1}$  and the worst case 
happens only for $\xi$ in narrow conical regions along the coordinate
axes. This observation was implemented by making suitable localizations on the frequency space. An important ingredient, necessary to make these arguments rigorous in \cite{B3}, was Pisier's holomorpic semigroup theorem \cite{Pis1}. The arguments presented in \cite{B3} are based on a very explicit analysis which does not immediately carry over to other convex symmetric bodies. Therefore, new methods will need to be invented to understand the growth of $C_p(d, G)$, as $d\to\infty$, in inequality \eqref{eq:109} for general symmetric convex bodies $G\subset\RR^d$  when $p\in(1, 3/2]$.

\subsection{Weak type $(1,1)$ considerations}
So far we have only discussed  the question of dimension-free estimates on $L^p(\RR^d)$ spaces for $p\in(1, \infty]$. However, one may ask about a dimension-free bound for the best constant $C_1(d,G)$ in the weak type $(1,1)$ estimate
\begin{equation}
\label{eq:weak11}
\sup_{\lambda>0}\lambda\big|\big\{x\in \RR^d\colon \sup_{t>0}|M_t^G f(x)|>\lambda\big\}\big|
\le C_1(d,G)\|f\|_{L^1}.
\end{equation}
Appealing to the Vitali covering lemma one can easily show that $C_1(d,G)\le 3^d$. 
In \cite{SteinStro} the third author and Str\"omberg proved that for general symmetric convex bodies $G\subset\RR^d$ one has
\begin{equation}
\label{eq:StStr1}
C_1(d,G)\le Cd\log d,
\end{equation}
where $C>0$ is a universal constant independent of $d\in\NN$. This is the best known result to date, see also \cite{NT} for generalizations of \eqref{eq:StStr1}. The proof of inequality \eqref{eq:StStr1} is based on a rather complicated variant of the Vitali covering idea. The authors in \cite{SteinStro}  were also able to sharpen this estimate in the case of the Euclidean balls  by proving
\begin{equation}
\label{eq:StStr2}
C_1(d,B^2)\le Cd,
\end{equation}
with a universal constant $C>0$ independent of the dimension. For justifying \eqref{eq:StStr2} the authors used a comparison with the heat semigroup together with the Hopf maximal ergodic theorem, see \cite{Ste1}.

Now, in view of these results a natural question arises, whether we can take a dimension-free constant in \eqref{eq:StStr1} and \eqref{eq:StStr2}. This was resolved in the case of the cube $G=B^{\infty}$ by Aldaz \cite{Ald1} who proved that
\begin{equation}
\label{eq:Ald1}
C_1(d,B^{\infty})\ge C_d,
\end{equation}
where $C_d$ is a constant that tends to infinity as $d\to \infty$. The constant $C_d$  was made more explicit by Aubrun \cite{Aub1}, who proved \eqref{eq:Ald1} with $C_d\simeq_{\varepsilon}(\log d)^{1-\varepsilon}$ for every $\varepsilon>0$,  and by Iakolev and Str\"omberg \cite{IakStr1}, who considerably improved the latter lower bound by showing that  $C_d\simeq d^{1/4}$. The arguments in the papers \cite{Ald1}, \cite{Aub1}, \cite{IakStr1}, were based on careful analysis of a discretized version of the initial problem. The function $f$ realizing the supremum was then chosen as an appropriate sum of Dirac's deltas.     

The case of the cube is the only one where we have a definitive answer on the size of $C_1(d,G)$ in \eqref{eq:weak11}. Remarkably even in the case of the Euclidean ball $B^2$ it is unknown whether the weak type $(1,1)$ constant is dimension-free.

\section{Overview of the methods of the paper}
\label{sec:3}
This section is intended to  present a new flexible approach, which recently resulted in dimension-free bounds in $r$-variational and jump inequalities corresponding to the operators $M_t^G$ from  \eqref{eq:43}, see \cite{BMSW1} and \cite{MSZ1}. An important feature of this method is that it is also applicable to the discrete settings, see \cite{BMSW3,MSZ1}.

For clarity of exposition  we shall only be working with maximal functions on $L^p(\RR^d)$ or $\ell^p(\ZZ^d)$. For a more abstract setting we refer to \cite{MSZ1}, and also \cite{MSZ0}.

\subsection{Continuous perspective}
We shall briefly  outline the method of the proof of Theorem \ref{thm:4}. The proof of \eqref{eq:43} is based on the following simple decomposition 
\begin{align}
\label{eq:20}
\sup_{t>0}|M_t^Gf|\le\sup_{n\in\ZZ}|M_{2^n}^Gf|
+\Big(\sum_{n\in\ZZ}\sup_{t\in[2^n, 2^{n+1}]}|(M_t^G-M_{2^n}^G)f|^2\Big)^{1/2}.
\end{align}
In other words, the full maximal function corresponding to the operators $M_t^G$  is controlled by the dyadic maximal function and the square function associated with maximal functions restricted to dyadic blocks.

The estimates, on $L^p(\RR^d)$ for $p\in(1, \infty]$,  of the dyadic maximal function (in fact inequality \eqref{eq:44}) are based, upon comparing  $\sup_{n\in\ZZ}|M_{2^n}^Gf|$ with the Poisson semigroup $P_t$, see \eqref{eq:31},  on a variant of bootstrap argument. The idea of the bootstrap goes back to \cite{NSW}, where the context of differentiation in lacunary directions was studied. Later on, these ideas were used in many other papers \cite{Car2}, \cite{DuoRu}, including their applications in dimension-free estimates \cite{Car1}. Recently, it turned out that certain variant of bootstrap arguments may be also used to obtain dimension-free estimates in $r$-variational inequalities \cite{BMSW1,BMSW3} and in jump inequalities \cite{MSZ1}. In the latter paper applications to the operators of Radon type are discussed as well. The methods of  \cite{MSZ1}, presented as a part of an abstract theory, immediately give the desired conclusion. However, in Section \ref{sec:5}, for the sake of clarity, we give a simple direct proof and deduce \eqref{eq:44} from inequality \eqref{eq:27}, which immediately leads to a bootstrap inequality in  \eqref{eq:32}. In particular three tools, with dimension-free estimates, that we now highlight are used to obtain \eqref{eq:27}:
\begin{enumerate}[label*={\arabic*}.]
\item The maximal inequality \eqref{eq:47'} for the Poisson semigroup $P_t$.
\item The Littlewood--Paley inequality \eqref{eq:6} associated with the Poisson projections $S_n$.
\item The estimates of the Fourier multiplies corresponding to $M_t^G$ from Theorem \ref{prop:1}.
\end{enumerate}
The details are given in the second part of Section \ref{sec:5}. In the third part of Section \ref{sec:5} we estimate, 
on $L^p(\RR^d)$ for $p\in(3/2, 2]$, the square function from \eqref{eq:20}. In order to do so, we shall employ  an elementary numerical inequality, as in \cite{BMSW1, BMSW3}, see also \cite{MSZ1}, which asserts that for every
$n\in\ZZ$ and for every function
$\mathfrak a:[2^n, 2^{n+1}]\to\CC$  we have
\begin{align}
\label{eq:24}
\sup_{t\in[2^n, 2^{n+1}]}|\mathfrak a(t)-\mathfrak a(2^n)|
&\le \sqrt{2} \sum_{l\in\NN_0} \Big( \sum_{m = 0}^{2^{l}-1} \big|\mathfrak a(2^n+{2^{n-l}(m+1)})
-\mathfrak a(2^n+{2^{n-l}m})\big|^2 \Big)^{1/2}.
\end{align}

The inequality is the crucial new ingredient, which on the one hand, replaces the fractional integration argument from \cite{Car1}.
This  is especially important in the discrete setting as it is not clear, due to the lack of the dilation structure on $\ZZ^d$, whether the fractional integration argument is available there.
On the other hand, \eqref{eq:24} reduces estimates for a supremum (or even for $r$-variations, see \cite{MSZ1}) restricted to a dyadic block to the situation of certain square functions, where the division intervals over which differences are taken (in these square functions) are all of the same size, see inequality  \eqref{eq:23}. 

A variant of  inequality \eqref{eq:24} was proved by Lewko--Lewko \cite[Lemma 13]{LL}, and it was used to study variational Rademacher--Menshov type results for orthonormal systems. Inequality \eqref{eq:24}, essentially in this form, was independently obtained in \cite[Lemma 1]{MT1} by the second author and Trojan  in the context of $r$-variational estimates for discrete Radon transforms, see also \cite{MTS1,MSZ2}.

Upon applying inequality \eqref{eq:24} to control the square function from \eqref{eq:20} the problem is reduced  to control a new square function like in \eqref{eq:34}. The problem now is well suited to an application of the Fourier transform methods, and the estimates from Theorem \ref{prop:1} combined with the Littlewood--Paley inequality do the job and we obtain the desired claim.

The approach described above does not allow us to improve the range for $p\in(3/2, \infty]$ in the inequality from \eqref{eq:43}. To see this, it suffices to consider  the maximal function corresponding to the spherical means in $\RR^3$, see \eqref{eq:42}. Indeed, adopting the method from Section \ref{sec:5}  we obtain that the spherical maximal function is bounded on $L^p(\RR^3)$ for every $p\in(3/2, \infty]$, but unbounded on $L^{3/2}(\RR^3)$, see \cite{Ste0}.    

Any extension of the range $p\in(3/2, \infty]$ in \eqref{eq:43} will require more refined information besides the positivity of the operators $M_t^G$ and estimates of the Fourier multipliers $m_t^G$ from Theorem \ref{prop:1}. To be more precise, assume that $p_0\in(1, 2]$ and let $\alpha=1/p_0<1$. Suppose that there is a constant $C_{p_0}>0$ independent of the dimension $d\in\NN$ such that for every $t>0$ and $h\in(0, 1)$ and for every $f\in L^{p_0}(\RR^d)$ the following H\"older continuity condition holds 
\begin{align}
\label{eq:2}
\|(M_{t+h}^G-M_t^G)f\|_{L^{p_0}}\le C_{p_0}\bigg(\frac{h}{t}\bigg)^{\alpha}\|f\|_{L^{p_0}}.
\end{align}
Then, as it was proved in \cite{MSZ1} using a certain bootstrap argument, for every $p\in(p_0, 2]$ we have
\begin{align}
\label{eq:3}
C_p(d, G)\lesssim_{p}1,
\end{align}
with the implicit constant independent of the dimension. Therefore, the general problem is reduced to understand \eqref{eq:2}. In the case of $q$-balls $G=B^q$ for $q\in[1, \infty]$, inequality \eqref{eq:2}, and consequently \eqref{eq:3}, can be verified as it was shown in \cite{BMSW1,MSZ1}. The general case is reduced, anyway, to understand the norm $\|T_{(\xi\cdot\nabla)^{\alpha}m^G}\|_{L^p\to L^p}$ as in M\"uller's proof \cite{Mul1}. But, as we said before, this will  need new ideas.

\subsection{Discrete perspective}
As we have seen in the introduction the dimension-free estimates in the discrete setting for $\mathcal C_p(d, G)$ may be very hard, and in general there is no obvious conjecture to prove. However, for the $q$-balls $G=B^q$ as in \eqref{eq:88}, in view of the methods presented above, the problem may be reduced to  estimates of the Fourier multipliers.
For $q\in[1, \infty]$, let $\mathfrak m_N^{B^q}$ be the multiplier corresponding to the operator $\mathcal M_N^{B^q}$ as in \eqref{eq:85}. Let us define the proportionality factor
\begin{align*}
%\label{eq:50}
\kappa_q(d, N)=Nd^{-1/q},
\end{align*}
which can be identified with the isotropic constant corresponding to $B^q_N$, if the normalization assumption $|B^q|=1$ in definition \eqref{eq:iso} is dropped.   If we could prove that there exists a constant $C_q>0$ independent of the dimension $d\in\NN$ such that for every $N\in\NN$ and $\xi\in\TT^d$ we have
        \begin{align}
        \label{eq:49}
        \begin{split}
|\mathfrak m_N^{B^{q}}(\xi)-1|&\le C_q\kappa_q(d, N)|\xi|,\\
         |\mathfrak m_N^{B^{q}}(\xi)|&\le C_q(\kappa_q(d, N)|\xi|)^{-1},\\
          |\mathfrak m_{N+1}^{B^{q}}(\xi)-\mathfrak m_{N}^{B^{q}}(\xi)|&\le C_qN^{-1},
          \end{split}
          \end{align}
where $|\xi|$ denotes the Euclidean norm restricted to the torus $\TT^d\equiv[-1/2, 1/2)^d$; then,  using the methods from the proof of Theorem \ref{thm:4}, we would be able to conclude that the best constant $\mathcal C_p(d, B^q)$ in inequality \eqref{eq:86} is bounded independently of the dimension for every $p\in(3/2, \infty]$. 

Therefore, the problem of estimating $\mathcal C_p(d, B^q)$ with bounds independent of the dimension is reduced to establishing \eqref{eq:49}. Even though, estimates \eqref{eq:49} can be thought of as discrete analogues of the estimates for the continuous multipliers $m_t^{G}$, from Theorem \ref{prop:1} with $G=B^q$, the method of the proof of Theorem \ref{prop:1} is not applicable to derive \eqref{eq:49}.  For $q\in[1, \infty)$ the question seems to be very hard due to the lack of reasonable estimates for the number of lattice points in the sets $B^q_N$.

However, if $q=\infty$ then $B^{\infty}_N=[-N, N]^d$ is a cube. Thus  the number of lattice points is not a problem any more, and we easily have $|B^{\infty}_N\cap\ZZ^d|=(2N+1)^d$. This property distinguishes the cubes from the  $q$-balls for $q\in[1, \infty)$. Using the product structure of the cubes  we were able to analyze  the behavior of the multiplier
$\mathfrak m_N^{B^{\infty}}$ associated with the operator $\mathcal
M_N^{B^{\infty}}$ and obtain \eqref{eq:49}, see \cite{BMSW3} for more details. The multiplier
$\mathfrak m_N^{B^{\infty}}$ is an exponential sum, which is the product of
one dimensional Dirichlet's kernels. The explicit formula for
$\mathfrak m_N^{B^{\infty}}$ in terms of the Dirichlet kernels was essential for our approach  and permitted us to establish \eqref{eq:49} for $q=\infty$ with  $\kappa_{\infty}(d, N)=N$. Applying \eqref{eq:49} we showed in \cite{BMSW3}, as it was mentioned in the introduction,  that for every $p\in(3/2, \infty]$ there is a constant $C_p>0$ independent of the dimension such that $\mathcal C_p(d, B^{\infty})\le C_p$. Moreover, if the supremum in \eqref{eq:86} is restricted to the dyadic set $\mathbb D$, then \eqref{eq:86} holds for all $p\in(1, \infty]$ and  $\mathcal C_p(d, B^{\infty})$ is independent of the dimension as well. The inequalities in \eqref{eq:49}, for $q=\infty$, are based on elementary estimates, which are interesting in their own right. For this reason our method does not extend to discrete convex bodies other than $B^{\infty}$. This is the second place which sets the operators $\mathcal M_N^{B^{\infty}}$ over the cubes apart from the operators $\mathcal M_N^{B^q}$ over the $q$-balls for $q\in[1, \infty)$.

Now it is desirable to understand whether inequalities \eqref{eq:49} hold for $q\in[1, \infty)$. The absence  of
the product structure for $q\in[1, \infty)$ makes the estimates incomparably harder. However, using crude estimates for the number of lattice points in the $q$-balls $B^q_N$, if $p\in(1, \infty]$ and $q\in[1, \infty]$, we obtain, as in \cite{BMSW3}, that there is  $C_{p,q}>0$ independent of the dimension $d\in\NN$ such that  for all $f\in\ell^p(\ZZ^d)$ we have
\begin{align}
\label{eq:4}
\big\|\sup_{N\ge d^{1+1/q}}|\mathcal M_N^{B^q} f|\big\|_{\ell^p}
\leq C_{p,q} \|f\|_{\ell^p}.
\end{align}

Inequality \eqref{eq:4} follows from a simple comparison argument, which permits us to dominate the $\ell^p(\ZZ^d)$ norm   of the maximal function $\sup_{N\ge d^{1+1/q}}|\mathcal M_N^{B^q} f|$ by a constant multiple of $C_p(d, B^q)$, which we know is independent of the dimension for every $p\in(1, \infty]$ due to \cite{Mul1} for $q\in[1, \infty)$, and due to \cite{B3} for $q=\infty$.

In Section \ref{sec:6}, for $q=2$, we shall extend the range in the supremum in \eqref{eq:4} and we show that  $d^{1+1/q}=d^{3/2}$, (for $q=2$), can be replaced by a constant multiple of $d$, see  Theorem \ref{thm:3}. Our argument is a subtle refinement of the arguments from \cite{BMSW3}. Even though, we will also use crude estimates for the number of lattice points in the balls $B_N^2$, the essential improvement comes from the fact that the Euclidean norm corresponds to the scalar product $|x|^2=\langle x, x\rangle$. See Lemma \ref{lem:1} and Lemma \ref{lem:2}, where this observation plays the key role. The rest of the argument reduces the problem to the comparison of the $\ell^p(\ZZ^d)$ norm of $\sup_{N\ge Cd}|\mathcal M_N^{B^2}f|$ with $C_p(d, B^2)$, which is independent of the dimension for all $p\in(1, \infty]$.
Now the matters are reduced to understand $\sup_{1\le N\le Cd}|\mathcal M_N^{B^2}f|$.

In \cite{BMSW2} the authors initiated investigations in this direction and the case of the discrete Euclidean balls with dyadic radii was studied. We obtained Theorem \ref{thm:0}, which gives us some evidence that inequality \eqref{eq:86} with dimension-free bounds in not entirely hopeless, at least for $q=2$. The methods of the proof of Theorem \ref{thm:0} shed a new light on the  general problem \eqref{eq:86}, but the best what we can do for the full maximal function at this moment is Theorem \ref{thm:3}, and new methods will surely need to be  invented to attack this case.

The proof of Theorem \ref{thm:0} is based on the the estimates for $\mathfrak m_N^{B^2}$, which in turn are based on delicate combinatorial arguments that differ completely from the methods used to obtain  estimates \eqref{eq:49} for $\mathfrak m_N^{B^{\infty}}$. In particular, we proved analogues of the first two inequalities from \eqref{eq:49} for  $\mathfrak m_t^{B^{2}}$. However, the second inequality is perturbed by a negative power of $\kappa_2(d, N)$, which makes our method limited to the dyadic scales, and nothing reasonable beyond $\ell^2(\ZZ^d)$ theory can be said in \eqref{eq:90}. Our aim now is to understand whether the second estimate can be improved. If we succeeded in doing so, we could extend inequality \eqref{eq:90} to $\ell^p(\ZZ^d)$ spaces  for all $p\in(1, \infty]$. The second task, which seems to be quite challenging, is to obtain  the third inequality in \eqref{eq:49} for the multiplier $\mathfrak m_N^{B^{2}}$. This inequality, if proved, would allow us to think about dimension-free estimates of $\mathcal C_p(d, B^2)$ for all $p\in(3/2, \infty]$. We refer to \cite{BMSW2} for more details.

\section{Continuous perspective:  proof of Theorem \ref{thm:4}}
\label{sec:5}
The purpose of this section is to provide dimension-free estimates on $L^p(\RR^d)$,  with $p\in(3/2, \infty]$, for the Hardy--Littlewood maximal function  associated with convex symmetric bodies in $\RR^d$. However, we begin with the proof of Theorem \ref{prop:1}, which will allow us to build up the $L^2(\RR^d)$ theory in Theorem \ref{thm:4}.  

\subsection{Fourier transform estimates: proof of Theorem \ref{prop:1}}

For $\zeta\in \mathbb S^{d-1}$ and $u\in\RR$ we define the set 
\[
A_{\zeta}(u)=\{x\in G\colon x\cdot \zeta =u\},
\]
and an even and
compactly supported function by setting
\begin{align*}
\varphi^G_{\zeta}(u)={\rm Vol}_{d-1}(A_{\zeta}(u)),
\end{align*}
where ${\rm Vol}_{d-1}$ denotes $(d-1)$-dimensional Lebesgue
measure. We observe that for all $\lambda\in[0, 1]$ and for all
$u, v\in\RR$ such that $\varphi^G_{\zeta}(u)\not=0$ and
$\varphi^G_{\zeta}(v)\not=0$ we obtain
\begin{align}
\label{eq:52}
{\rm Vol}_{d-1}(A_{\zeta}(\lambda u +(1-\lambda) v))^{\frac{1}{d-1}}\ge
\lambda{\rm Vol}_{d-1}(A_{\zeta}(u))^{\frac{1}{d-1}}+(1-\lambda) {\rm Vol}_{d-1}(A_{\zeta}(v))^{\frac{1}{d-1}}.
\end{align}
This can be verified using Brunn--Minkowski's inequality (in
dimension $(d-1)$),
since, by convexity of $G$,  for every $u,v\in\RR$ if $A_{\zeta}(u)\not=\emptyset$ and
$A_{\zeta}(v)\not=\emptyset$ then 
\begin{align}
\label{eq:51}
\la A_{\zeta}(u)+(1-\lambda)A_{\zeta}(v)\subseteq A_{\zeta}(\lambda u +(1-\lambda) v).
\end{align}
For $\zeta\in \mathbb S^{d-1}$ define
$S_{\zeta}=\{x\in\RR\colon \varphi_\zeta^G(x)\not=0\}$. If
$u_0\in S_{\zeta}^c$ then, using \eqref{eq:51}, it is not difficult to
see that for every $u\in\RR$ such that $|u|>|u_0|$ we have
$\varphi_\zeta^G(u)=0$. This ensures that $S_{\zeta}$ is a symmetric
interval contained in $[-u_{\zeta}, u_{\zeta}]$, where
$u_{\zeta}=\sup\{x\ge0\colon\varphi_\zeta^G(x)\not=0\}$. Taking $v=-u$
in \eqref{eq:52} we obtain that
$\varphi_\zeta^G((2\lambda-1)u)\ge \varphi_\zeta^G(u)$ for all
$\lambda\in[0, 1]$ and $u\in\RR$. This implies that $\varphi_\zeta^G$
is decreasing on $S_{\zeta}\cap[0,\infty)$ as well as on $[0,\infty)$.
Inequality \eqref{eq:52} shows that the function
$(\varphi_\zeta^G)^{\frac{1}{d-1}}$ is concave on $S_{\zeta}$. In
particular, $\varphi_\zeta^G$ is differentiable almost everywhere in
$(-u_{\zeta}, u_{\zeta})$, since it is absolutely continuous on each
closed interval contained in $(-u_{\zeta}, u_{\zeta})$.  The
inequality between the weighted arithmetic and geometric means
together with \eqref{eq:52} also implies the log-concavity of
$\varphi_\zeta^G$. Namely, for $\lambda\in[0, 1]$ and $u,v>0$ we have
\begin{equation*}
%\label{eq:53}
\varphi_\zeta^G(\lambda u + (1-\lambda)v)\ge \varphi_\zeta^G(u)^{\lambda}\varphi_\zeta^G(v)^{1-\lambda}.
\end{equation*}
Note that using Fubini's theorem we have, for $\xi\in \RR^d\setminus\{0\}$, that
\begin{equation}
\label{eq:mGsect}
m^G(\xi)=\int_{\RR}\varphi^G_{\xi/|\xi|}(u)e^{2\pi i |\xi| u}{\rm d}u.\end{equation}
More generally, for any $h\in L^{\infty}(\RR)$ and $\xi \in \RR^d\setminus\{0\}$, one has 
\begin{equation}
\label{eq:mGsecth}
\int_G h(x\cdot\xi){\rm d}x=\int_{\RR}\varphi^G_{\xi/|\xi|}(u)h(|\xi| u){\rm d}u.
\end{equation}

From the above properties of $\varphi^G_{\zeta}$ we shall deduce, as in \cite[Lemma 1]{B1}, that 
\begin{equation}
\label{eq:fiinf}
\varphi^G_{\zeta}(u)\le 2\varphi^G_{\zeta}(0)e^{- \varphi^G_{\zeta}(0)|u|},\quad\text{for all}\quad u\in \RR, \quad\text{and}\quad\zeta\in\mathbb S^{d-1}.
\end{equation}
For this purpose, we fix $\zeta\in\mathbb S^{d-1}$ and let us consider the function
$\psi^G_{\zeta}(u)=\varphi^G_{\zeta}(0)e^{- \varphi^G_{\zeta}(0)|u|}$,
whose logarithm is a linear function. We have that $\varphi^G_{\zeta}(0)=\psi^G_{\zeta}(0)$, and suppose that 
there is a point $u_0\in(0, \infty)$ such that
$\varphi^G_{\zeta}(u_0)=\psi^G_{\zeta}(u_0)$. By the
log-concavity we obtain that
\begin{align*}
\varphi^G_{\zeta}(u)\le \psi^G_{\zeta}(u), \quad \text{for }\quad u>u_0,
\end{align*}
and in this case there is nothing to do. Moreover, the log-concavity also gives 
\begin{align*}
\varphi^G_{\zeta}(u)\ge \psi^G_{\zeta}(u), \quad \text{for }\quad 0\le u\le u_0.
\end{align*}
In this case, using \eqref{eq:mGsecth} with $h(u)=\ind{[0,\infty)}(u)$ we obtain
\begin{align}
\label{eq:61}
  \frac12=\int_0^{\infty}\varphi^G_{\zeta}(u){\rm d}u
  \ge \varphi^G_{\zeta}(0)\int_{0}^{u_0}e^{-u\varphi^G_{\zeta}(0)}{\rm d}u
  =\int_{0}^{u_0 \varphi^G_{\zeta}(0)}e^{-u}{\rm d}u=1-e^{-u_0\varphi^G_{\zeta}(0)},
  \end{align} 
  and, consequently, $e^{-u_0\varphi^G_{\zeta}(0)}\ge 1/2$, so that
  $u_0\varphi^G_{\zeta}(0)\le \log2$. Hence, \eqref{eq:fiinf} follows,
  since
  \begin{align*}
  \varphi^G_{\zeta}(u)\le \varphi^G_{\zeta}(0)e^{(u_0-u)\varphi^G_{\zeta}(0)}
  \le 2\varphi^G_{\zeta}(0)e^{-u\varphi^G_{\zeta}(0)},\quad\text{for}\quad 0\le u\le u_0.
  \end{align*}
  If $u=0$ is the unique point such that
  $\varphi^G_{\zeta}(0)=\psi^G_{\zeta}(0)$, then
  $\varphi^G_{\zeta}(u)\le \psi^G_{\zeta}(u)$ or
  $\varphi^G_{\zeta}(u)\ge \psi^G_{\zeta}(u)$ for all
  $u\in S_{\zeta}$. If the first inequality holds then we are done, so
  we may assume that the second inequality is true. Arguing in a
  similar way as in \eqref{eq:61} with $u_{\zeta}$ in place of $u_0$
  we obtain that $u_{\zeta}\varphi^G_{\zeta}(0)\le \log2$, and consequently
  \begin{align*}
  \varphi^G_{\zeta}(u)\le \varphi^G_{\zeta}(0)e^{(u_{\zeta}-u)\varphi^G_{\zeta}(0)}
  \le 2\varphi^G_{\zeta}(0)e^{-u\varphi^G_{\zeta}(0)},\quad\text{for}\quad 0\le u\le u_{\zeta}.
  \end{align*}
  Hence \eqref{eq:fiinf} follows, since $\varphi^G_{\zeta}(u)=0$ for $u\in S_{\zeta}^c$.

 Since $G$ is in the isotropic position we can also prove that $\varphi^G_{\zeta}(0)$ is of the same order, uniformly in $\zeta \in \mathbb S^{d-1}.$ More precisely, as in \cite[Lemma 2]{B1}, we have  
\begin{equation}
\label{eq:fiL}
\frac{3}{16}\le L \,\varphi^G_{\zeta}(0)\le 3,\quad\text{for every}\quad \zeta\in \mathbb S^{d-1},
\end{equation}  
where $L$ is the isotropic constant. To prove the right-hand side inequality in \eqref{eq:fiL} we show, with the aid of \eqref{eq:mGsecth} (for $h(u)=u^2$) and \eqref{eq:fiinf}, that
\begin{align*}
L^2=\int_{\RR} u^2 \varphi^G_{\zeta}(u){\rm d}u
\le 4 \varphi^G_{\zeta}(0)\int_{0}^{\infty}u^2 e^{-\varphi^G_{\zeta}(0)u}{\rm d}u
\le 8\varphi^G_{\zeta}(0)^{-2}. 
\end{align*}
For the left-hand side  inequality in \eqref{eq:fiL} we calculate
\begin{align*}
1=\int_{ \RR}\varphi^G_{\zeta}(u){\rm d}u
\le 4L\varphi^G_{\zeta}(0)+\frac{1}{4L^2}\int_{ |u|\ge 2L} u^2\varphi^G_{\zeta}(u){\rm d}u
\le 4L \varphi^G_{\zeta}(0)+\frac 14,
\end{align*}
which implies $L\varphi^G_{\zeta}(0)\ge 3/16$, and \eqref{eq:fiL} is justified.
We now pass to the proof of Theorem \ref{prop:1}.

\begin{proof}[Proof of Theorem \ref{prop:1}]
We begin with the proof of inequalities in \eqref{eq:69}. For $\xi\in\RR^d\setminus\{0\}$ we set
$\zeta=\xi/|\xi|$, then integration by parts
allows us to rewrite \eqref{eq:mGsect} as
\begin{align*}
m^G(\xi)=\int_{\RR}\varphi_{\zeta}^G(u)\cos(2\pi |\xi| u){\rm d}u
=\lim_{u\nearrow u_{\zeta}}\frac{\varphi_{\zeta}^G(u)\sin(2\pi |\xi|u)}{\pi|\xi|}
-\frac{1}{2\pi|\xi|}\int_{-u_{\zeta}}^{u_{\zeta}}(\varphi_{\zeta}^G)'(u)\sin(2\pi |\xi|u){\rm d}u.
\end{align*}
Then using \eqref{eq:fiL} we obtain the first inequality in  \eqref{eq:69}, since
\begin{align*}
|m^G(\xi)|&\le(\pi|\xi|)^{-1}\varphi_{\zeta}^G(0)+ (2\pi|\xi|)^{-1} \int_{-u_{\zeta}}^{u_{\zeta}}|(\varphi_{\zeta}^G)'(u)|{\rm d}u\\
&=(\pi|\xi|)^{-1}\varphi_{\zeta}^G(0)-(\pi|\xi|)^{-1}\int_{0}^{u_{\zeta}}(\varphi_{\zeta}^G)'(u){\rm d}u\\
&\le 6\pi^{-1} (L|\xi|)^{-1}.
\end{align*}
To prove the second inequality in  \eqref{eq:69}, we use \eqref{eq:fiinf} and \eqref{eq:fiL} to write
\begin{align*}
|m^G(\xi)-1|&\le\int_{\RR}\varphi_{\zeta}^G(u) |\cos(2\pi |\xi|u) -1|{\rm d}u\\
&\le  4\pi |\xi| \int_{0}^{\infty}u\varphi_{\zeta}^G(u){\rm d}u\\
&\le 8\pi|\xi|\varphi_{\zeta}^G(0)^{-1} \\
&\le 45\pi(L|\xi|).
\end{align*}
This completes the proof of \eqref{eq:69}. To justify \eqref{eq:69'}, we use \eqref{eq:mGsecth} and integrate by parts to get
\begin{align*}
\langle\xi,\nabla m^G(\xi)\rangle
&=\int_G2\pi i  \langle x, \xi\rangle e^{2\pi i x\cdot \xi}{\rm d}x\\
&=\int_{-u_{\zeta}}^{u_{\zeta}} (2\pi i|\xi|e^{2\pi i u|\xi|})\,(u\varphi_{\zeta}^G(u)){\rm d}u\\
&=\lim_{u\nearrow u_{\zeta}}\big(e^{2\pi i u|\xi|}u\varphi_{\zeta}^G(u)\big)-\lim_{u\searrow -u_{\zeta}}\big(e^{2\pi i u|\xi|}u\varphi_{\zeta}^G(u)\big) -\int_{-u_{\zeta}}^{u_{\zeta}}e^{2\pi i u|\xi|}\frac{{\rm d}}{{\rm d}u}(u\varphi_{\zeta}^G(u)){\rm d}u.
\end{align*}
This leads, in view of \eqref{eq:fiinf}, to the estimate
\begin{align*}
|\langle\xi,\nabla m^G(\xi)\rangle|&\le 4u_{\zeta}\varphi^G_{\zeta}(0)e^{- \varphi^G_{\zeta}(0)u_{\zeta}}
+\int_{-u_{\zeta}}^{u_{\zeta}} \varphi_{\zeta}^G(u){\rm d}u
+\int_{-u_{\zeta}}^{u_{\zeta}} |u||(\varphi_{\zeta}^G)'(u)|{\rm d}u\\
&\le5-2\int_0^{u_{\zeta}}u(\varphi_{\zeta}^G)'(u){\rm d}u,
\end{align*}
where we used the fact that $\varphi_{\zeta}^G(u)$ is decreasing in $u$.
Hence, integrating by parts once again we reach $|\langle\xi,\nabla m^G(\xi)\rangle|\le 10$, which gives \eqref{eq:69'}.
The proof of Theorem \ref{prop:1} is completed. 
\end{proof}

The approach we shall use to prove Theorem \ref{thm:4} was presented as a part of an abstract theory in \cite{MSZ1}. The method has recently  found many applications in $r$-variational and jump estimates (including dimension-free estimates) in the continuous and discrete settings, see \cite{BMSW1,BMSW3,MSZ0,MSZ1}. However here, for the sake of clarity, we shall only focus our attention on the maximal functions in the continuous setup.

Since we are working with a family of averaging operators only the range for $p\in(3/2, 2]$ will be interesting in Theorem \ref{thm:4}. The range for $p\in(2, \infty]$ will follow then by a simple interpolation with the obvious $L^{\infty}(\RR^d)$ bound. For instance, in order to prove dimension-free bounds for the dyadic maximal function, it will suffice to show that for every $p\in(1, 2]$ and for every $f\in L^p(\RR^d)$ we have 
\begin{align}
\label{eq:25}
\big\|\sup_{n\in\ZZ}|M^G_{2^n}f|\big\|_{L^p}
\lesssim\|f\|_{L^p}.
\end{align}
In particular, \eqref{eq:25} proves inequality \eqref{eq:44} from Theorem \ref{thm:4}.
Then, in view of \eqref{eq:20}, the proof of inequality \eqref{eq:43} will be completed, if we show that for $p\in(3/2, 2]$ and for every $f\in L^p(\RR^d)$ we have 
\begin{align}
\label{eq:26}
\Big\|\big(\sum_{n\in\ZZ}\sup_{t\in[2^n, 2^{n+1}]}|(M_t^G-M_{2^n}^G)f|^2\big)^{1/2}\Big\|_{L^p}\lesssim\|f\|_{L^p}.
\end{align}
In the next two subsections we prove inequalities \eqref{eq:25} and \eqref{eq:26} respectively.

\subsection{Proof of inequality \eqref{eq:25}}
We fix $N\in\NN$ and define
\begin{align*}
B_p(N)=\sup_{\|f\|_{L^p}\le1}\big\|\sup_{|n|\le N}|M^G_{2^n}f|\big\|_{L^p}.
\end{align*}
We see that $B_p(N)\le2N+1$ for every $N\in\NN$, since $M_t^G$ is an averaging operator. Our aim will be to show that for every $p\in(1, 2]$ there is a constant $C_p>0$ independent of the dimension and the underlying  body $G\subset\RR^d$ such that
\begin{align}
\label{eq:33}
\sup_{N\in\NN}B_p(N)\le C_p.
\end{align}
Observe that, by \eqref{eq:47'}, we have
\begin{align}
\label{eq:31}
\begin{split}
\big\|\sup_{|n|\le N}|M^G_{2^n}f|\big\|_{L^p}&
\le \big\|\sup_{t>0}|P_{t}f|\big\|_{L^p}
+\big\|\sup_{|n|\le N}|(M^G_{2^n}-P_{2^n})f|\big\|_{L^p}\\
&\lesssim \|f\|_{L^p}+\sum_{j\in\ZZ}\Big\|\big(\sum_{|n|\le N}|(M^G_{2^n}-P_{2^n})S_{j+n}f|^2\big)^{1/2}\Big\|_{L^p},
\end{split}
\end{align}
where in the last line we have used decomposition from \eqref{eq:5}.
The proof of \eqref{eq:25} will be completed, if we show that for every $p\in(1, 2]$ there is $C_p'>0$ independent of $d$, $N$, and the body $G\subset\RR^d$ such that for every $j\in\ZZ$ and for every $f\in L^p(\RR^d)$ we have
\begin{align}
\label{eq:27}
\Big\|\big(\sum_{|n|\le N}|(M^G_{2^n}-P_{2^n})S_{j+n}f|^2\big)^{1/2}\Big\|_{L^p}\le C_p' (1+B_p(N))^{\frac{2-p}{2}}2^{-\frac{(p-1)|j|}{2}}\|f\|_{L^p}.
\end{align}
Assume momentarily that \eqref{eq:27} has been proven. Then combining \eqref{eq:31} with \eqref{eq:27} we obtain that
\begin{align}
\label{eq:32}
B_p(N)\lesssim_p 1+(1+B_p(N))^{\frac{2-p}{2}},
\end{align}
with the implicit constant independent of $d$, $N$ and the body $G\subset\RR^d$. Thus we conclude, using \eqref{eq:32}, that \eqref{eq:33} holds, and the proof of \eqref{eq:25} and consequently  \eqref{eq:44} from Theorem \ref{thm:4} is completed.
\subsubsection{Proof of inequality \eqref{eq:27} for $p=2$} Using Theorem \ref{prop:1} we show that \eqref{eq:27} holds  for $p=2$. Let $k(\xi)=m^G(\xi)-p_1(\xi)=m^G(\xi)-e^{-2\pi L|\xi|}$ be the multiplier associated with the operator $M_1^G-P_1$.  Observe that by Theorem \ref{prop:1} and the properties of $p_1(\xi)$ there exists a constant $C>0$ independent of the dimension and the body $G\subset\RR^d$ such that
  \begin{align}
\label{eq:30}
    |k(\xi)|\le C\min\big\{L|\xi|, (L|\xi|)^{-1}\big\},
    \end{align}
    where $L=L(G)$ is the isotropic constant as in \eqref{eq:iso}.
Now by \eqref{eq:30} and Plancherel's theorem we
get
\begin{align}
\label{eq:29}
  \begin{split}
  \Big\|\big(\sum_{n\in\ZZ}|(M^G_{2^n}&-P_{2^n})S_{j+n}f|^2\big)^{1/2}\Big\|_{L^2}\\
&=\Big(\int_{\RR^d}\sum_{n\in\ZZ}
\big|k(2^n\xi)\big(e^{-2\pi2^{n+j} L|\xi|}-e^{-2\pi2^{n+j-1} L|\xi|}\big)\big|^2
  |\mathcal Ff(\xi)|^2{\rm d}\xi\Big)^{1/2}\\
  &\lesssim2^{-|j|/2}\Big(\int_{\RR^d}\sum_{n\in
  \ZZ}\min\big\{2^nL|\xi|, (2^nL|\xi|)^{-1}\big\}
  |\mathcal Ff(\xi)|^2{\rm d}\xi\Big)^{1/2}\\
  &\lesssim 2^{-|j|/2}\|f\|_{L^2},
  \end{split}
  \end{align}
  with the implicit constant  independent of $d$, $N$ and the body $G\subset\RR^d$.
This proves \eqref{eq:27} for $p=2$. 
\subsubsection{Proof of inequality \eqref{eq:27} for $p\in(1, 2)$}
For $s\in(1, 2]$ and $r\in[1, \infty]$, let  $A_N(s, r)$ be the smallest constant in the following inequality
    \begin{align}
\label{eq:22}
\Big\|\big(\sum_{|n|\le N}|(M^G_{2^n}&-P_{2^n})g_n|^r\big)^{1/r}\Big\|_{L^s}
\le A_N(s, r)\Big\|\big(\sum_{|n|\le N}|g_n|^r\big)^{1/r}\Big\|_{L^s}.
    \end{align}
    It is easy to see that $A_N(s, r)<\infty$. Let $u\in(1, p)$ be such that $\frac{1}{u}=\frac{1}{2}+\frac{1}{2p}$.
Now it is not difficult to see that $A_N(1, 1)\lesssim 1$, since $\|(M^G_{2^n}-P_{2^n})f\|_{L^p}\leq 2\|f\|_{L^p}$. Moreover, by \eqref{eq:47'}, if $g=\sup_{|n|\le N}|g_n|$ then
\begin{align*}
  \big\|\sup_{|n|\le N}|(M^G_{2^n}-P_{2^n})g_n|\big\|_{L^p}\lesssim (B_p(N)+1)\|g\|_{L^p}.
\end{align*}
Hence by the complex interpolation we obtain
$$A_N(u, 2)\le A_N(1, 1)^{1/2}A_N(p, \infty)^{1/2}\lesssim (B_p(N)+1)^{1/2}.$$
Then by \eqref{eq:22} and \eqref{eq:6} we get
\begin{align}
\label{eq:28}
  \begin{split}
\Big\|\big(\sum_{|n|\le N}|(M^G_{2^n}-P_{2^n})S_{j+n}f|^2\big)^{1/2}\Big\|_{L^u}
&\le A_N(u, 2)\Big\|\big(\sum_{n\in\ZZ}|S_{j+n}f|^2\big)^{1/2}\Big\|_{L^u}\\
&\lesssim (B_p(N)+1)^{1/2}\|f\|_{L^u}.
  \end{split}
\end{align}
We now take  $\rho\in(0, 1]$ satisfying $\frac{1}{p}=\frac{1-\rho}{u}+\frac{\rho}{2}$,
then $\rho=p-1$ and  $1-\rho=2-p$. Interpolation between \eqref{eq:29} and \eqref{eq:28} yields
\eqref{eq:27} for $p\in(1, 2)$ as desired.

\subsection{Proof of inequality \eqref{eq:26}}
To estimate \eqref{eq:26} we use \eqref{eq:5} and \eqref{eq:24} and obtain
\begin{align}
\label{eq:23}
\begin{split}
\Big\|\big(\sum_{n\in\ZZ}\sup_{t\in[2^n, 2^{n+1}]}|(M_t^G&-M_{2^n}^G)f|^2\big)^{1/2}\Big\|_{L^p}\\
&\lesssim
\sum_{l\ge 0}\sum_{j\in\ZZ} \Big\|\big(\sum_{n\in\mathbb Z} \sum_{m = 0}^{2^{l}-1}\abs{(M^G_{2^n+{2^{n-l}(m+1)}} - M^G_{2^n+{2^{n-l}m}})S_{j+n}f}^2 \big)^{1/2}\Big\|_{L^p}.
\end{split}
\end{align}
Our aim now is to show that for every $q\in(1, 2)$ and $\theta\in[0, 1]$ such that  $\frac{1}{p}=\frac{\theta}{2}+\frac{1-\theta}{q}$ we have, for every  $f\in L^p(\RR^d)$, the following estimate
\begin{align}
\label{eq:34}
\begin{split}
\Big\|\big(\sum_{n\in\mathbb Z} \sum_{m = 0}^{2^{l}-1}
|(M^G_{2^n+{2^{n-l}(m+1)}} &- M^G_{2^n+{2^{n-l}m}})S_{j+n}f|^2 \big)^{1/2}\Big\|_{L^p}\\
&\lesssim 2^{-\theta l/2+(1-\theta)l}\min\big\{1, 2^{l}2^{-|j|/2}\big\}^{\theta}\|f\|_{L^p},
\end{split}
\end{align}
with the implicit constant independent of the dimension and the underlying  body $G\subset\RR^d$.

Assume momentarily that \eqref{eq:34} has been proven. Then we combine \eqref{eq:23} with \eqref{eq:34} and obtain estimate \eqref{eq:26}, since the double series  
\begin{align*}
\sum_{l\ge 0}\sum_{j\in\ZZ}2^{-\theta l/2+(1-\theta)l}\min\big\{1, 2^{l}2^{-|j|/2}\big\}^{\theta}\lesssim 1
\end{align*}
is summable, whenever $\theta /2-(1-\theta)>0$, which forces $p$ to satisfy $\frac{3}{1+1/q}< p\le 2$, due to $\theta=\frac{2}{p}\frac{p-q}{2-q}$. This completes the proof of \eqref{eq:43} from Theorem \ref{thm:4}.

\subsubsection{Proof of inequality \eqref{eq:34} for $p=2$} Using  inequalities  \eqref{eq:69} and arguing in a similar way as in \eqref{eq:29} we obtain 
\begin{align}
\label{eq:35}
\Big\|\big(\sum_{n\in\mathbb Z} \sum_{m = 0}^{2^{l}-1}\abs{(M^G_{2^n+{2^{n-l}(m+1)}} - M^G_{2^n+{2^{n-l}m}})S_{j+n}f}^2 \big)^{1/2}\Big\|_{L^2}\lesssim2^{l/2}2^{-|j|/2}\|f\|_{L^2}.
\end{align}
Note that inequality  \eqref{eq:69'} implies
\begin{align*}
|m^G((2^n+{2^{n-l}(m+1)})\xi)-m^G((2^n+{2^{n-l}m})\xi)|
\le\int_{2^n+{2^{n-l}m}}^{2^n+{2^{n-l}(m+1)}}|\langle t\xi,\nabla m^G(t\xi)\rangle|\frac{{\rm d}t}{t}
\lesssim 2^{-l}.
\end{align*}
Therefore, by Plancehrel's theorem 
\begin{align}
\label{eq:37}
\begin{split}
\Big\|\big(\sum_{n\in\mathbb Z} \sum_{m = 0}^{2^{l}-1}|(M^G_{2^n+2^{n-l}(m+1)} &- M^G_{2^n+2^{n-l}m})S_{j+n}f|^2 \big)^{1/2}\Big\|_{L^2}\\
&=\Big(\sum_{n\in\mathbb Z} \sum_{m = 0}^{2^{l}-1}\|(M^G_{2^n+{2^{n-l}(m+1)}} - M^G_{2^n+{2^{n-l}m}})S_{j+n}f\|_{L^2}^2\Big)^{1/2}\\
&\lesssim \Big(\sum_{n\in\mathbb Z} 2^{-l}\|S_{j+n}f\|_{L^2}^2\Big)^{1/2}\\
&\lesssim 
2^{-l/2}\|f\|_{L^2}.
\end{split}
\end{align}
Combining \eqref{eq:35} and \eqref{eq:37} we obtain
\begin{align}
\label{eq:38}
\Big\|\big(\sum_{n\in\mathbb Z} \sum_{m = 0}^{2^{l}-1}\abs{(M^G_{2^n+{2^{n-l}(m+1)}} - M^G_{2^n+{2^{n-l}m}})S_{j+n}f}^2 \big)^{1/2}\Big\|_{L^2}\lesssim 2^{-l/2}\min\big\{1, 2^{l}2^{-|j|/2}\big\}\|f\|_{L^2},
\end{align}
which proves \eqref{eq:34} for $p=2$.

\subsubsection{Proof of inequality \eqref{eq:34} for $p\in(3/2, 2)$}
We begin with a general remark,  a consequence of \eqref{eq:44}, which states that for every $q\in(1, \infty)$ there is a constant $C_q>0$ 
independent of the dimension and the underlying  body $G\subset\RR^d$ such that for every sequence $(g_n)_{n\in\ZZ}\in L^q(\ell^2(\RR^d))$ we have 
\begin{align}
\label{eq:40}
\Big\|\big(\sum_{n\in\mathbb Z} |M^G_{2^n}g_n|^2 \big)^{1/2}\Big\|_{L^q}
\le C_q
\Big\|\big(\sum_{n\in\mathbb Z} |g_n|^2 \big)^{1/2}\Big\|_{L^q}.
\end{align}
Indeed, let $A(q, r)$ be the best constant in the following inequality
\begin{align*}
\Big\|\big(\sum_{n\in\mathbb Z} |M^G_{2^n}g_n|^r \big)^{1/r}\Big\|_{L^q}
\le A(q, r)
\Big\|\big(\sum_{n\in\mathbb Z} |g_n|^r \big)^{1/r}\Big\|_{L^q}.
\end{align*}
By the complex interpolation and duality ($A(q, r)=A(q', r')$) and inequality \eqref{eq:44} we obtain
\begin{align*}
A(q, 2)\le A(q, 1)^{1/2}A(q, \infty)^{1/2}=A(q', \infty)^{1/2}A(q, \infty)^{1/2}\le C_{q'}^{1/2}C_{q}^{1/2},
\end{align*}
which implies \eqref{eq:40}.
Observe that by \eqref{eq:40} and \eqref{eq:6}, since $M^G_{2^n(1+t)}=M^{(1+t)G}_{2^n}$,  we obtain
\begin{align}
\label{eq:39}
\begin{split}
\Big\|\big(\sum_{n\in\mathbb Z} \sum_{m = 0}^{2^{l}-1}
|(M^G_{2^n+{2^{n-l}(m+1)}} &- M^G_{2^n+{2^{n-l}m}})S_{j+n}f|^2 \big)^{1/2}\Big\|_{L^q}\\
&\lesssim 2^l\sup_{t\in[0, 1]}\Big\|\big(\sum_{n\in\mathbb Z} 
|M^G_{2^n(1+t)}S_{j+n}f|^2 \big)^{1/2}\Big\|_{L^q}\\
&\lesssim 2^l\Big\|\big(\sum_{n\in\mathbb Z} 
|S_{j+n}f|^2 \big)^{1/2}\Big\|_{L^q}\\
&\lesssim 2^l\|f\|_{L^q}.
\end{split}
\end{align}
Interpolating \eqref{eq:38} with \eqref{eq:39} we obtain \eqref{eq:34} as desired.

\section{Discrete perspective:  proof of Theorem \ref{thm:3}}
\label{sec:6}
The main objective of this section is to provide dimensional-free estimates on $\ell^p(\ZZ^d)$, for $p\in(1, \infty]$, of the norm of the  maximal function corresponding to the operators $\mathcal M_N^{B^2}$ from \eqref{eq:85} with
large scales $N\ge Cd$ for some $C>0$, where $N>0$ is a real number.  The estimate in \eqref{eq:7} will be deduced by comparison of
$\sup_{N\ge Cd}|\mathcal M_N^{B^2}f|$ with its continuous analogue, for which we have dimension-free bounds
provided by the third author in \cite{SteinMax}.
Namely, we know that
for every $p\in(1,\infty)$ there is $C_p>0$ independent of the
dimension such that for every $f\in L^p(\RR^d)$ we have
        \begin{align}
\label{eq:8}
          \|M_*^{B^2}f\|_{L^p}\le C_p\|f\|_{L^p}.
          \end{align}
          
Throughout this section, unless otherwise stated, $N>0$ is always a real number and $Q=[-1/2,1/2]^d$ denotes the unit cube.  A fundamental role, in the proofs of this section, will be played by the  fact that the Euclidean norm corresponds to the scalar product $|x|^2=\langle x, x\rangle$. We begin with crude estimates for the number of lattice points the Euclidean balls $B^2_N$.

\begin{lemma}
	\label{lem:1}
Let $N>0$  and set $N_1=(N^2+d/4)^{1/2}$. Then
	\begin{equation}
\label{eq:9}
	|B^2_N\cap\ZZ^d|\le 2 |B^2_{N_1}|.
	\end{equation}
        Moreover, if $N\ge Cd$ for some fixed  $C>0$,
	then  we have
	\begin{equation}
	\label{eq:10}
	|B^2_N\cap\ZZ^d| \le 2e^{{1}/{(8C^2)}} |B^2_{N}|.
	\end{equation}
\end{lemma}
\begin{proof}
For $x\in B^2_N$ and $z\in Q$ we have  
\[
|x+z|^2\le N^2+\frac{d}{4}+2\langle x, z\rangle.
\]
Moreover, for all $x\in B^2_N$ we have
\[
|\{z\in Q\colon \langle x, z\rangle \le 0\}|\ge \frac12.
\]
Hence
\begin{align*}
  |B^2_N\cap\ZZ^d|&=\sum_{x\in B^2_N\cap\ZZ^d}1\leq 2\sum_{x\in
                  B^2_N\cap\ZZ^d} \int_{Q}\ind{\{z\in Q\colon \langle x, z\rangle \le
                  0\}}(y){\rm d}y\\
                &\le 2\sum_{x\in B^2_N\cap\ZZ^d} \int_{Q}\ind{\{z\in Q\colon |x+z|\le N_1\}}(y){\rm d}y\\
                &\le 2\sum_{x\in \ZZ^d} \int_{Q}\ind{B^2_{N_1}}(x+y){\rm
                  d}y\\
  &=2\sum_{x\in \ZZ^d}\int_{x+Q}\ind{B^2_{N_1}}(y){\rm d}y=2|B^2_{N_1}|.
\end{align*}
This proves \eqref{eq:9}. For \eqref{eq:10} note that for $N\ge Cd$ we get
\begin{align*}
|B^2_{N_1}|=\frac{\pi^{d/2}N_1^d}{\Gamma(d/2+1)}=|B^2_N|\bigg(1+\frac{d}{4N^2}\bigg)^{d/2}
\le|B^2_N|\bigg(1+\frac{1}{4CN}\bigg)^{d/2}\le e^{{1}/{(8C^2)}}|B^2_N|,
\end{align*}
which proves \eqref{eq:10}.
\end{proof}

\begin{lemma}
	\label{lem:2}
Assume that $N\ge Cd$ for some fixed $C>0$ and let $t>0$.  Then for every $x\in \RR^d$
such that $|x|\ge N(1+t/N)^{1/2}$ we  have
	\begin{equation}
\label{eq:11}
|Q\cap(B^2_N-x)|=|\{y\in Q\colon x+y\in B^2_N\}|\le 2e^{-ct^2},
	\end{equation}
        where $c=\frac{7}{32}\frac{C^2}{(C+1)^2}$.
\end{lemma}
\begin{proof}
Let $|x|\ge N(1+t/N)^{1/2}.$ Then	for $y\in Q$ and $x+y\in B^2_N$ we have
	\begin{equation*}
	N^2+Nt\le|x|^2=|x+y-y|^2\le N^2-2\langle x, y\rangle-|y|^2\le N^2+2|\langle x, y\rangle|.
	\end{equation*}
	Thus for $\bar{x}=x/|x|$ one has
	\begin{equation*}
	|\langle \bar{x}, y\rangle|\ge \frac12 \frac{Nt}{|x|}\ge\frac12 \frac{Nt}{|x+y|+|y|}\ge \frac12 \frac{Nt}{N+d^{1/2}}\ge \frac12 \frac{Ct}{C+1},
	\end{equation*}
	and consequently  we get
        \begin{equation}
        \label{eq:12-}
        |\{y\in Q\colon x+y\in B^2_N\}|
        \le|\{y\in Q: |\langle \bar{x}, y\rangle|
        \ge Ct/(2C+2)\}|.
        \end{equation} 
        We claim that 
         for every unit vector $z\in\RR^d$ and
        for every $s>0$  we have
        \begin{align}
          \label{eq:12}
          |\{y\in Q: \langle z, y\rangle\ge
        s\}|\le e^{-\frac78s^2}. 
        \end{align}
        Taking $s=Ct/(2C+2)$ in \eqref{eq:12} and coming back to \eqref{eq:12-} we complete the proof of \eqref{eq:11} with $c=\frac{7}{32}\frac{C^2}{(C+1)^2}.$

        In the proof of \eqref{eq:12} we will appeal to the inequality $e^{x}+e^{-x}\le 2e^{\frac12
          x^2}$, which holds for all $x\ge0$. 
        Indeed, for every $\alpha>0$ we get  
        \begin{align*}
          e^{\alpha s}|\{y\in Q: \langle z, y\rangle\ge s\}|
          &\le \int_Qe^{\alpha
            \sum_{j=1}^dz_j y_j}{\rm d}y\\
          &=\prod_{j=1}^d\int_0^{1/2}e^{\alpha
            z_j y_j}+e^{-\alpha
            z_j y_j}{\rm d}y_j\\
          &\le\prod_{j=1}^d2\int_0^{1/2}e^{\frac12\alpha^2
            (z_j y_j)^2}{\rm d}y_j\\
          &\le\prod_{j=1}^de^{\frac18\alpha^2
          z_j^2}\\
          &=e^{\frac18\alpha^2
          \sum_{j=1}^dz_j^2}\\
          &= e^{\frac18\alpha^2}.
        \end{align*}
        Taking $\alpha=s$ in the inequality above and
        dividing by $e^{s^2}$ we obtain
        \eqref{eq:12} and the proof is completed.
\end{proof}

\begin{lemma}
	\label{lem:3}
There are constants $C_1, C_2>0$ such that for every $N\ge C_1d$ we have
	\begin{equation}
	\label{eq:13}
	|B^2_N|\le C_2|B^2_{N}\cap\ZZ^d|.
	\end{equation}
        Moreover,
       \eqref{eq:13} combined with \eqref{eq:10} from Lemma
       \ref{lem:1} yields 
	\begin{equation*}
	%\label{eq:14}
	C_2^{-1}|B^2_N|\le |B^2_{N}\cap\ZZ^d|\le 2e^{{1}/{(8C_1^2)}} |B^2_{N}|
	\end{equation*}
	for every $N\ge C_1d$.

\end{lemma}
\begin{proof}
We show that there is $J\in\NN$ such that for every $M\ge d$ we have
\begin{align}
  \label{eq:18}
  |B^2_M|\le 2|B^2_{M(1+J/M)^{1/2}}\cap\ZZ^d|.
\end{align}
Assume momentarily that \eqref{eq:18} is proven, then \eqref{eq:13}
follows. Indeed, for every $N\ge C_1d$, where $C_1=2(1+J)$ we find $M\ge d$ such that
$N=M(1+J/M)^{1/2}$, hence, \eqref{eq:18} implies
\begin{align}
  \label{eq:19}
  |B^2_M|\le 2|B^2_{N}\cap\ZZ^d|.
\end{align}
On the other hand we have
\[
|B^2_M|\le |B^2_N|\le (1+J/M)^{d/2}|B^2_M|\le e^{J}|B^2_M|,
\]
since $M\ge d$. This estimate combined with \eqref{eq:19} gives
\eqref{eq:13} with $C_2=2e^J$.

Our aim now is to prove \eqref{eq:18}. For this purpose let $J\in\NN$ be a large number such that
\begin{align*}
  %\label{eq:17}
\sum_{j\ge J}e^{-j^2/32}e^{j}\le \frac{1}{8e}.
\end{align*}
Define  $U_j=\big\{x\in\RR^d: M\big(1+\frac{j}{M}\big)^{1/2}< |x|\leq
M\big(1+\frac{(j+1)}{M}\big)^{1/2}\big\}$ and observe that
\begin{align}
  \label{eq:15}
  \begin{split}
  |B^2_M|
  &=\sum_{x\in\ZZ^d}\int_{x+Q}\ind{B^2_M}(y){\rm d}y\\
  &=\sum_{x\in\ZZ^d}\int_{Q}\ind{B^2_M}(x+y){\rm d}y\\
&\le\sum_{x\in B^2_M\cap\ZZ^d}\int_{Q}\ind{B^2_M}(x+y){\rm
  d}y+\sum_{j\ge0}\sum_{x\in U_j\cap\ZZ^d}\int_{Q}\ind{B^2_M}(x+y){\rm d}y\\
&\le |B^2_M\cap\ZZ^d|+\sum_{0\le j<J}\sum_{x\in U_j\cap\ZZ^d}|Q\cap
(B^2_M-x)|+\sum_{j\ge J}\sum_{x\in U_j\cap\ZZ^d}|Q\cap (B^2_M-x)|\\
&\le |B^2_{M(1+J/M)^{1/2}}\cap\ZZ^d|+\sum_{j\ge J}\sum_{x\in U_j\cap\ZZ^d}|Q\cap (B^2_M-x)|.
  \end{split}
\end{align}

By \eqref{eq:10}, since $M\ge d$, we get 
\begin{align*}
  |B^2_{M(1+(j+1)/M)^{1/2}}\cap\ZZ^d|&\le2e^{1/8}|B^2_{M(1+(j+1)/M)^{1/2}}|\\
  &\le 2e^{1/8}\bigg(1+\frac{j+1}{d}\bigg)^{d/2}|B^2_M|\\
  &\le 2e^{1/8} e^{(j+1)/2}|B^2_M|.
\end{align*}
Using this estimate, the definition of the sets $U_j$ and Lemma \ref{lem:2} we
obtain for any $M\ge d$ that
\begin{align}
  \label{eq:16}
  \begin{split}
  \sum_{j\ge J}\sum_{x\in U_j\cap\ZZ^d}|Q\cap (B^2_M-x)|
  &\le 2\sum_{j\ge J}e^{-j^2/32}|B^2_{M(1+(j+1)/M)^{1/2}}\cap\ZZ^d|\\
  &\le 4e^{5/8}|B^2_M|\sum_{j\ge J}e^{-j^2/32}e^{j}\\
  &\le \frac{1}{2}|B^2_M|.
  \end{split}
\end{align}
Combining \eqref{eq:16} with \eqref{eq:15} we obtain
\eqref{eq:18} as desired. This completes the proof of Lemma \ref{lem:3}.
\end{proof}

We  now are ready to prove Theorem \ref{thm:3}.
\begin{proof}[Proof of Theorem \ref{thm:3}]
	Let $f: \ZZ^d\to \CC$ and define its extension $F: \RR^d\to
        \CC$ on $\RR^d$ by setting
        \[
        F(x)=\sum_{y\in\ZZ^d}f(y)\ind{y+Q}(x).
        \]
	Then, clearly $\|F\|_{L^p(\RR^d)}=\|f\|_{\ell^p(\ZZ^d)}$ for
        every $p\ge1$.

        From now on we assume that $f\ge0$.
        For every $N\ge C_1d$, with $C_1$ as in Lemma
        \ref{lem:3}, we define  $N_1=(N^2+d/4)^{1/2}$. 
        Observe that for $z\in Q$ and $y\in B^2_{N}$ we have
        \[
|y+z|^2= |y|^2+|z|^2+2\langle z, y\rangle\le N_1^2
\]
on the set $\{z\in Q\colon \langle z, y\rangle\le 0\}$, which has measure $1/2$.
        Then by Lemma \ref{lem:3}  for all
        $x\in\ZZ^d$ we obtain
        \begin{align}
        \label{eq:63}
\begin{split}
\mathcal M_N^{B^2}f(x)&
=\frac1{|B^2_N\cap\ZZ^d|}\sum_{y\in B^2_N\cap\ZZ^d}f(x+y)\ind{B^2_N}(y)\\
&\lesssim \frac1{|B^2_N|}\sum_{y\in\ZZ^d}f(x+y)\int_{Q}\ind{B^2_{N_1}}(y+z){\rm d}z\\
&=\frac1{|B^2_N|}\sum_{y\in\ZZ^d}f(y)\int_{x+B^2_{N_1}}\ind{y+Q}(z){\rm d}z\\
&=\frac1{|B^2_N|}\int_{x+B^2_{N_1}}F(z){\rm d}z\\
&=\bigg(\frac{N_1}{N}\bigg)^d\frac{1}{|B^2_{N_1}|}\int_{B^2_{N_1}}F(x+z){\rm d}z\\
&\lesssim\frac{1}{|B^2_{N_1}|}\int_{B^2_{N_1}}F(x+z){\rm d}z\\
&=M_{N_1}^{B^2}F(x).        
\end{split}
\end{align}	
Finally, take $N_2=(N_1^2+d/4)^{1/2}.$
Similarly as above, for  $y\in Q$ and $z\in B^2_{N_1}$ we have 
$$|y+z|^2\le |y|^2+|z|^2+2\langle z, y\rangle\le N_2^2$$
on the set $\{y\in Q\colon \langle z, y\rangle\le 0\}$, which has Lebesgue measure $1/2$. 
Therefore, Fubini's theorem leads to
\begin{align}
\label{eq:64}
\begin{split}
M_{N_1}^{B^2}F(x)&=\frac{1}{|B^2_{N_1}|}\int_{B^2_{N_1}}F(x+z){\rm d}z\\
&\leq \frac2{|B^2_{N_1}|}
\int_{\RR^d}F(x+z)\ind{B^2_{N_1}}(z)\int_{Q}\ind{B^2_{N_2}}(z+y){\rm d}y{\rm d}z\\
&\lesssim \frac1{|B^2_{N_2}|}\int_{Q}\int_{\RR^d}F(x+z-y)\ind{B^2_{N_2}}(z){\rm d}z{\rm d}y\\
&= \int_{x+Q}M_{N_2}^{B^2}F(y){\rm d}y.
\end{split}
	\end{align}
        Combining \eqref{eq:63} with \eqref{eq:64}, applying H\"older's inequality, and invoking \eqref{eq:8} we arrive at
	\begin{align*}
          \big\|\sup_{N\ge C_1 d}|\mathcal M_N^{B^2}f|\big\|_{\ell^p(\ZZ^d)}^p
          &\lesssim\sum_{x\in\ZZ^d}\int_{x+Q}\big|\sup_{N\ge
            C_1 d}M_N^{B^2}F(y)\big|^p{\rm d}y\\
          &=\big\|\sup_{N\ge C_1 d}M_N^{B^2}F\big\|_{L^p(\RR^d)}^p\\
          &\lesssim\|F\|_{L^p(\RR^d)}^p\\
          &=\|f\|_{\ell^p(\ZZ^d)}^p.
        \end{align*}
        This proves Theorem \ref{thm:3} with $C=C_1.$ 
\end{proof}


\begin{thebibliography}{99}
\bibitem{Ald1} J.M. Aldaz, \textit{The weak type (1, 1) bounds for the
maximal function associated to cubes grow to infinity with the
dimension}. Ann.  Math.  {\bf 173}, no. 2, (2011), pp. 1013--1023.

\bibitem{Aub1} G. Aubrun, \textit{Maximal inequality for
high-dimensional cubes}. Confluentes Math. {\bf 1} (2009), pp. 169--179.

\bibitem{B0} J. Bourgain, \textit{Averages in the plane over convex
curves and maximal operators }. J. Analyse Math. {\bf 46}, no. 1,
(1986), pp. 69--85.

\bibitem{B33} J. Bourgain, \textit{On dimension free maximal
inequalities for convex symmetric bodies in {${\mathbf R}^n$}},
Geometric Aspects of Functional Analysis, Israel Seminar (GAFA) (1985/86). Lecture Notes in
Mathematics, vol.  {\bf 1267}, (1987), pp. 168--176. Springer, Berlin, Heidelberg.

\bibitem{B1} J. Bourgain, \textit{On high dimensional maximal
functions associated to convex bodies}. Amer. J.  Math. {\bf 108},
(1986), pp. 1467--1476.
        
\bibitem{B2} J. Bourgain, \textit{On $L^p$ bounds for maximal
functions associated to convex bodies in $\RR^n$}. Israel J. Math.
 {\bf 54}, (1986), pp. 257--265.

\bibitem{B00} J. Bourgain, \textit{On the distribution of polynomials
on high dimensional convex sets}. Geometric Aspects of Functional
Analysis, Israel Seminar (GAFA) (1989/90). Lecture Notes in
Mathematics, vol. {\bf 1469}, (1991), pp. 127--137. Springer, Berlin,
Heidelberg.


\bibitem{B3} J. Bourgain, \textit{On the Hardy-Littlewood maximal
function for the cube}. Israel J. Math.  {\bf 203}, (2014),
pp. 275--293.


\bibitem{BMSW1} J. Bourgain, M. Mirek, E. M. Stein, B. Wr\'obel,
\textit{Dimension-free variational estimates on $L^p(\RR^d)$ for
symmetric convex bodies}. Geom. Funct. Anal.  {\bf 28}, no. 1, (2018),
pp. 58--99.
  

\bibitem{BMSW3} J. Bourgain, M. Mirek, E. M. Stein, B. Wr\'obel,
\textit{Dimension-free estimates for discrete Hardy--Littlewood
averaging operators over the cubes in
$\mathbb Z^d$}. Amer. J. Math. {\bf 141}, (2019), no. 4, pp. 857--905.


\bibitem{BMSW2} J. Bourgain, M. Mirek, E. M. Stein, B. Wr\'obel,
\textit{On discrete Hardy--Littlewood maximal functions over the balls
in $\ZZ^d$: dimension-free estimates}. To appear in the Geometric
Aspects of Functional Analysis. Israel Seminar (GAFA). Lecture Notes
in Mathematics, Springer.
  
\bibitem{BGVV} S. Brazitikos, A. Giannopoulos, P. Valettas, B.-H.Vritsiou,
\textit{Geometry of Isotropic Convex Bodies}.
Mathematical Surveys and Monographs, American Mathematical Society, 2014, pp. 1--594.

\bibitem{Car1} A. Carbery, \textit{An almost-orthogonality principle
with applications to maximal functions associated to convex bodies}.
Bull. Amer. Math. Soc.  {\bf 14} no. 2 (1986), pp. 269--274.

\bibitem{Car2} A. Carbery, \textit{Differentiation in lacunary
directions and an extension of the Marcinkiewicz multiplier
theorem}. Ann. Inst. Fourier (Grenoble) {\bf 38} no. 1 (1988),
pp. 157--168.

\bibitem{Cow} M. G. Cowling,
\textit{Harmonic Analysis on Semigroups}.
Ann. Math. {\bf 117} no. 2 (1983), pp. 267--283.

\bibitem{DGM1} L. Delaval, O. Gu\'edon, B. Maurey,
\textit{Dimension-free bounds for the Hardy-Littlewood maximal
operator associated to convex sets}. Ann. Fac. Sci. Toulouse
Math. {\bf 27} no.1, (2018), 1--198.

\bibitem{DuoRu} J. Duoandikoetxea, J.L. Rubio de Francia,
\textit{Maximal and singular integral operators via Fourier transform
estimates}. Invent. Math.  {\bf 84}, no. 3 (1986), pp. 541--561.




\bibitem{IakStr1} A. S. Iakovlev and J. O. Str\"omberg, \textit{Lower
bounds for the weak type $(1, 1)$ estimate for the maximal function
associated to cubes in high dimensions}. Math. Res. Lett. {\bf 20} (2013),
pp. 907--918.


\bibitem{Kla1} B. Klartag, \textit{On convex perturbations with a
bounded isotropic constant}. Geom. Funct. Anal. (GAFA) {\bf 16} no. 6,
(2006), pp. 1274--1290.

\bibitem{LL} A. Lewko, M. Lewko \textit{Estimates for the square
variation of partial sums of Fourier series and their rearrangements}.
J. Funct.  Anal. {\bf 262}, no. 6, (2012), pp. 2561--2607.




\bibitem{MTS1} M. Mirek, E. M. Stein, B. Trojan,
\textit{$\ell^p(\ZZ^d)$-estimates for discrete operators of Radon
type: Variational estimates}. Invent. Math.  {\bf 209}, no. 3 (2017),
pp. 665--748.


\bibitem{MSZ0} M. Mirek, E.M. Stein, P. Zorin--Kranich.
\textit{Jump inequalities via real interpolation}.
To appear in the Mathematische Annalen (2019), \texttt{https://arxiv.org/abs/1808.04592}.


\bibitem{MSZ1} M. Mirek, E. M. Stein, P. Zorin--Kranich, \textit{A
bootstrapping approach to jump inequalities and their applications}.
To appear in the Analysis \& PDE (2019), \texttt{https://arxiv.org/abs/1808.09048}.

\bibitem{MSZ2} \textsc{M. Mirek, E.M. Stein, P. Zorin--Kranich.}
\textit{Jump inequalities for translation-invariant operators of Radon type on $\ZZ^d$}.
{Preprint (2018), \texttt{https://arxiv.org/abs/1809.03803}}.


\bibitem{MT1} M. Mirek, B. Trojan, \textit{ Discrete maximal functions
in higher dimensions and applications to ergodic theory}. Amer. J.
Math. {\bf 138}, (2016), no. 6, pp. 1495--1532.

\bibitem{Mul1} D. M\"uller, \textit{A geometric bound for maximal
functions associated to convex bodies}. Pacific J. Math. {\bf 142},
no. 2, (1990), pp. 297--312.

\bibitem{NSW} A. Nagel, E. M. Stein, S. Wainger,
\textit{Differentiation in lacunary
directions}. Proc. Nat. Acad. Sci. U.S.A.  {\bf 75}, no. 3, (1978),
pp. 1060--1062.

\bibitem{NT} A. Naor, T. Tao, \textit{Random martingales and
localization of maximal inequalities}. J. Funct. Anal. {\bf 259}, no. 3,
(2010), pp. 731--779.

\bibitem{Pis1} G. Pisier, \textit{Holomorphic semi-groups and the
geometry of Banach spaces}. Ann.  Math. {\bf 115} (1982), pp. 375--392.

\bibitem{Som} F. Sommer, \textit{Dimension free $L^p$-bounds of maximal
functions associated to products of Euclidean balls}. Preprint
(2018), \texttt{https://arxiv.org/abs/1703.07728}.

\bibitem{Ste0} E. M. Stein,
\textit{Maximal functions, spherical means}. Proc. Nat. Acad. Sci. U.S.A.
{\bf 73}, no. 7 (1976), pp. 2174--2175.  

\bibitem{Ste1} E. M. Stein, \textit{Topics in harmonic analysis
related to the Littlewood-Paley theory}. Annals of Mathematics
Studies, Princeton University Press 1970, pp. 1--157.

\bibitem{SteinMax} E. M. Stein, \textit{The development of square
functions in the work of A. Zygmund}. Bull.  Amer. Math. Soc.,
{\bf 7}, (1982), pp. 359--376.

\bibitem{SteinRiesz} E. M. Stein, \textit{Some results in harmonic
analysis in $\mathbb R^n$, $n\to \infty$}. Bull. Amer. Math. Soc.
{\bf 9}, (1983), pp. 71--73.


\bibitem{SteinStro} E. M. Stein, J. O. Str\"omberg, \textit{Behavior
of maximal functions in $\mathbb R^n$ for large $n$}. Ark. Mat.
{\bf 21}, no. 1-2, (1983), 259--269.

\bibitem{Zyg} A. Zygmund, \textit{Trigonometric series, Vols. I and II}.
Third edition. Cambridge Mathematical Library. Cambridge University
Press, Cambridge, 2002.

\end{thebibliography}
\end{document}